\documentclass[final]{siamltex}
\usepackage{graphicx}
\usepackage{subfigure}
\usepackage{amsfonts}
\usepackage{multirow}
\usepackage{amsmath}
\usepackage{amssymb}
\usepackage{hyperref}
\usepackage{caption}
\usepackage{microtype}
\usepackage{subfigure}
\usepackage{booktabs} 
\usepackage{mathtools}
\usepackage{indentfirst}
\usepackage{changepage}
\usepackage[capitalize,noabbrev]{cleveref}
\usepackage{multirow}
\usepackage{showkeys}
\usepackage{framed}
\usepackage{titlesec}
\usepackage{bm,amsfonts,amsbsy,amstext,amsopn,epstopdf,float,mathrsfs,booktabs}
\newtheorem{remark}[theorem]{Remark}
\usepackage{color}
\definecolor{Yingxiang}{rgb}{1, 0, 0}
\definecolor{other}{rgb}{0, 0, 1}

\title{Heterogeneous  optimized Schwarz Methods for heat conduction in composites with thermal contact resistance}

\author{Huan Zhang\thanks{School of Mathematics
    and Statistics, Northeast Normal University, Changchun 130024,
    China.({\tt zhangh514@nenu.edu.cn}).} 
    \and Hui Zhang\thanks{School of Mathematics and Physics, Xi'an Jiaotong-Liverpool University, Suzhou 215123,
    China. ({\tt Hui.Zhang@xjtlu.edu.cn}). }
    \and Yan Wang\thanks{School of Mathematics
    and Statistics, Northeast Normal University, Changchun 130024,
    China.({\tt wangy906@nenu.edu.cn}).}
    \and YingXiang Xu\thanks{Corresponding author. School of Mathematics
    and Statistics, Northeast Normal University, Changchun 130024,
    China. ({\tt yxxu@nenu.edu.cn}). }
    }
\begin{document}
\maketitle
\begin{abstract}
Heat transfer in composites is critical in engineering, where imperfect layer contact causes thermal contact resistance (TCR), leading to interfacial temperature discontinuity. We propose solving this numerically using the optimized Schwarz method (OSM), which decouples the heterogeneous problem into homogeneous subproblems. This avoids ill-conditioned systems from monolithic solving due to high contrast and interface jumps. Both energy estimate and Fourier analysis are used to prove the convergence of this algorithm when the standard Robin condition is applied to transmit information between subdomains.
To achieve fast convergence, instead of the standard Robin, the scaled Robin transmission condition is proposed, and the involved free parameter is rigorously optimized.
The results reveal several new findings due to the presence of TCR: first, the larger the TCR, the faster the OSM converges; second, mesh-independent convergence is achieved in the asymptotic sense, in contrast to the mesh-dependent results without TCR; and last, the heterogeneity contrast benefits the convergence, with a larger contrast leading to faster convergence. Interestingly, different from the case without TCR, the thermal conductivity also benefits the convergence, similar to the effect of heterogeneity. Numerical experiments confirm the theoretical findings and demonstrate the method's potential for nonlinear problems on irregular domains.
\end{abstract}


%\pagestyle{myheadings}
%\thispagestyle{plain}
%%%\markboth{}

\section{Introduction}
Composites are materials typically composed of various continuous solid components and are widely utilized in engineering applications such as aerospace \cite{mangalgiri1999composite} and electronics \cite{zweben1998advances} due to their superior physical properties and relatively low cost. 
The heat conduction processes within composites significantly influence their thermal conductivity, thermal stability, and mechanical properties. Consequently, thermal analysis of composites is crucial for simulating the manufacturing processes, thermal stresses, and welding of these materials, which have been extensively studied through detailed modeling to enhance both performance and cost-effectiveness. 

Steady-state heat conduction in composites is typically modeled using elliptic interface problems, a concept with a long historical lineage tracing back to \cite{peskin1977numerical,peskin1993improved}. 
Early research assumed that heat transfer between distinct solid materials was continuous; specifically, the temperature and heat flux at the interface were considered continuous \cite{frankel1987general}, a condition now recognized as perfect thermal contact. Notably, a substantial body of existing studies on heat transfer in composites relies on the assumption of perfect thermal contact. 
However, experimental evidence has demonstrated that temperature discontinuities frequently occur across interfaces \cite{hashin2001thin}. 
This phenomenon can be attributed to thermal contact resistance (TCR), also referred to as ``Kapitza'' resistance \cite{kapitza1941heat}, arising from imperfect contact between materials. Such imperfect contact is inevitable; the formation of interstitial media at rough contact interfaces limits the effective contact area, thereby introducing additional heat transfer resistance \cite{hasselman1987effective} compared to an ideal (perfect contact) scenario. This issue of imperfect contact is pervasive, occurring even in interface problems involving coupled solid and liquid phases. Therefore, it is imperative to consider TCR when analyzing heat transfer in composites. 

A considerable body of research has been conducted on solving elliptic interface problems, employing both analytical and numerical methods. 
Analytical methods primarily focus on the scenario of perfect contact, wherein both temperature and heat flux remain continuous. Notable examples include the separation of variables \cite{jain2010exact}, finite integral transforms \cite{singh2011finite}, Laplace transforms \cite{amiri2015exact}, Fourier transforms \cite{delouei2012exact}, and Green's functions \cite{cole2010heat}. However, research on analytical techniques addressing TCR is relatively limited \cite{haji2002temperature,haji2003steady}. 
The application of these analytical methods can be challenging due to their reliance on strict assumptions regarding the geometry of the domains and/or the boundary conditions. 
In contrast, numerical methods offer great adaptability in managing interface problems characterized by complex configurations, which have garnered significant attention  \cite{peskin1977numerical,leveque1994immersed,li2003new,gong2008immersed,oevermann2006cartesian,massjung2012unfitted}. When addressing problems with interfacial jump conditions, careful consideration is required to handle elements that intersect the interface.
Recently, several variants of the finite element method (FEM) have been developed to address imperfect contact problems, including the extended FEM \cite{yvonnet2011general,bakalakos2021extended} and unfitted FEM \cite{ji2017unfitted,wang2017weak}. Other numerical approaches, such as the finite volume method \cite{cao2018monotone} and the discontinuous Galerkin method \cite{cangiani2018adaptive}, have also been explored for tackling interface problems involving TCR. 
The aforementioned methods typically solve heat transfer problems in composite materials in a monolithic manner, which often lead to the formation of large, ill-conditioned algebraic systems due to heterogeneity contrasts. This issue complicates accurate problem resolution and frequently results in non-physical oscillations near the interface.
In contrast, domain decomposition methods \cite{smith1997domain,toselli2004domain} partition the computational domain into several homogeneous subdomains and reconstruct the solution through an iterative process of subdomain solving. This strategy mitigates the challenges associated with ill-conditioned large-scale algebraic systems and minimizes the occurrence of non-physical solutions. Notably, the presence of heterogeneity contrasts, which are detrimental to monolithic methods, can enhance the convergence rate of subdomain iterations \cite{gander2015optimized,gander2019heterogeneous}: the greater the heterogeneity contrast, the faster the subdomain iterations converge.
Furthermore, since the subproblems do not encounter the same heterogeneity challenges, existing solvers and codes can be directly employed to further augment the efficiency of the domain decomposition approach. These advantages render domain decomposition methods particularly well-suited for solving the heat transfer problems in composites, especially those involving TCR, as will be demonstrated in Section 3.

The first domain decomposition method was invented by Schwarz \cite{schwarz1870ueber} and is now commonly known as the Schwarz alternating method. This approach employs Dirichlet boundary conditions to facilitate information exchange between subdomains, necessitating an overlap between these subdomains to ensure convergence.
Later on, Lions \cite{lions1990schwarz} introduced a non-overlapping domain decomposition framework, wherein information transmission occurs via Robin boundary conditions. This advancement rendered non-overlapping domain decomposition methods viable for practical applications.
The Robin parameter employed in Lions' approach can be further optimized, resulting in the Optimized Schwarz Method (OSM), which is highly efficient since an optimization problem is solved to determine the best possible transmission parameters.
The OSM has been extensively utilized in various scientific and engineering contexts, including applications to the Navier-Stokes equation \cite{blayo2016towards}, Helmholtz equation \cite{gander2002optimized,gander2007optimized}, and Maxwell equation \cite{dolean2009optimized}. Notably, the non-overlapping OSM has been effectively applied to model problems characterized by discontinuous coefficients. In 2015, Dubois and Gander \cite{gander2015optimized} implemented the non-overlapping OSM for a steady-state thermal diffusion model featuring discontinuous coefficients, building on the foundational work of Maday and Magoulès \cite{maday2007optimized,maday2005non,maday2006improved}.
In 2019, Gander and Vanzan \cite{gander2019heterogeneous} applied the OSM to tackle heterogeneous second-order elliptic partial differential equations. 
The significant advantage that the large heterogeneity contrast improves the convergence rate of the OSM was observed, thus transforming a typical disadvantage of monolithic solving into a benefit of domain decomposition methods.

Though, existing domain decomposition algorithms have been successfully applied to solve problems with discontinuous coefficients, so far no domain decomposition methods have been designed to deal with solution jump across the interface.
The renowned Finite Element Tearing and Interconnecting method, developed in \cite{farhat1991method}, and the Balancing Domain Decomposition by Constraints approach, proposed in \cite{dohrmann2003preconditioner}, have been redesigned to address interface problems, as evidenced in \cite{pechstein2011analysis} and \cite{badia2019physics}. However, both methods require the continuity of the solution across the interface. Consequently, they are not suitable for solving heat transfer problems in composites, where the interface jump in the solution due to the TCR poses a challenge.
To fill this gap, we analyze the application of the OSM for solving the heat transfer in composites with TCR. The mathematical model under consideration is a two-dimensional steady-state diffusion-reaction equation, characterized by a continuous heat flux at the interface coupled with a temperature discontinuity proportional to the heat flux,  with the TCR as a factor.
This methodology can be readily extended to three-dimensional problems. 
Based on the characteristics of TCR, we first introduce the standard Robin condition for the OSM and establish the convergence results via both energy estimate and Fourier analysis. 
Toward fast convergence, the Robin parameter should be determined by optimizing the corresponding convergence factor, which is, however, not easy to accomplish. To achieve fast convergence, we propose the more analytically tractable scaled Robin transmission condition and systematically solve the optimization problem to determine the optimal parameters, thereby paving the way for practical applications.
The main findings are as follows.

1. The presence of TCR benefits the OSM algorithm, yielding mesh-\linebreak[4]independent convergence, a distinct contrast to the existing results for continuous interface conditions.

2. The large heterogeneity contrast improves the convergence of the OSM algorithm. 

3. In addition to heterogeneity contrast, the large thermal conductivities also accelerate the convergence of the OSM algorithm. This behavior fundamentally differs from the case without TCR.

4. The scaled Robin transmission condition not only is convenient for theoretical analysis, leading to an algorithm well-suited for practical applications, but also shows improved convergence over the standard Robin, as indicated by the numerical experiments.

The rest of this paper is organized as follows.
In Section 2, we first present the mathematical model incorporating TCR for describing the heat conduction in composites composed of two dissimilar materials. We then outline the OSM algorithm for addressing this problem and prove its convergence using both energy estimate and Fourier analysis. In Section 3, we introduce the scaled Robin transmission condition and optimize the transmission parameter by minimizing the convergence factor across all Fourier frequencies pertinent to the mesh. A comparison with existing results for the case without TCR reveals several noteworthy findings.
In Section 4, we conduct numerical experiments to illustrate the theoretical results. Finally, we draw conclusions in Section 5.

\section{Model problem and algorithm}
\label{sec:1}
\subsection{Model problem}
\label{sec:2}
We restrict our analysis to the scenario where the composite materials consist of two solid components, and the method can also be applied to composites of many-layered materials. 
The steady heat conduction in the composites can then be described by the following model
     \begin{equation}\label{2.1}
     \begin{aligned}
     -\nabla\cdot(\kappa(x) \nabla T)+c(x)T&=f \quad\mbox{in }\Omega,\\
      \mathcal{B}T&=g  \quad\mbox{on }\partial\Omega,
     \end{aligned}
     \end{equation}
where $\Omega:=\Omega_1\cup\Omega_2$ denotes a bounded open domain in $\mathbb{R}^d$ ($d=1,2,3$) consisting of the two bounded subdomains $\Omega_1$ and $\Omega_2$ occupied by different media, $\Omega_1\cap\Omega_2=\emptyset$. The interface between these two subdomains is defined as $\Gamma:=\partial{\Omega_1}\cap \partial{\Omega_2}$. 
The thermal conductivity $\kappa(x)$ and the reaction coefficient $c(x)$ are piecewise continuous functions
    \begin{align}\nonumber
     &
     \kappa(x)=\left\{
     \begin{aligned}
        &\kappa_{1}(x)   \quad \mbox{in}\ \Omega_1, \\
      &\kappa_{2}(x)   \quad \mbox{in}\ \Omega_2,  
     \end{aligned}
     \right.
     ~~~~~~~~~~
     c(x)=\left\{
     \begin{aligned}
        &c_{1}(x)   \quad \mbox{in}\ \Omega_1, \\
        &c_{2}(x)   \quad \mbox{in}\ \Omega_2,  
     \end{aligned}
     \right.
     \end{align}
where $\kappa_i(x)\geq \kappa^0>0$ and $c_i(x)\geq 0$ for $i=1,2$ are continuous, and $\kappa^0$ is a constant. 
Note that $c_i(x)>0$ can be associated with the reciprocal of the time step size when addressing time-dependent problems through a time-marching scheme. The function
$f$ represents a given source term, while $g$ denotes the specified boundary data relevant to the boundary operator $\mathcal{B}$ depending on application. 
The interface condition describes the fact that along the interface, the heat flux is continuous, while the difference in subdomain temperature is proportional to the heat flux, with the TCR serving as the proportionality coefficient. Mathematically, the interface condition can be expressed as:
\begin{equation}\label{2.2}
     \bm{n_{1}}\cdot(\kappa_{1}\nabla T_{1})=-\bm{n_{2}}\cdot(\kappa_{2}\nabla T_{2}), \quad 
       T_{1}-T_{2}=R_{c}\bm{n_{2}}\cdot(\kappa_{2}\nabla T_{2})
       \quad\mbox{on }\Gamma,
\end{equation}
where $R_c>0$ represents the TCR at the interface $\Gamma$, and $\bm{n_i}$ (for $i=1,2$) are the unit outward normal vectors corresponding to the subdomains $\Omega_i$. 
Notably, this interface condition  degenerates to the continuous case when $R_c=0$ and the problems have been solved with constant thermal conductivities and reaction coefficients in \cite{gander2019heterogeneous} for $c_2=0$ and in \cite{gander2015optimized} for $c_1=c_2=0$. 

Under the perfect contact scenario, the interface temperature remains continuous. However, the introduction of TCR in \eqref{2.2} disrupts this continuity, posing a significant challenge in the construction and analysis of the transmission conditions for the OSM. To address this issue, this paper, for the first time, proposes a two-domain OSM iterative scheme for the TCR-induced interfacial jump problem, aiming to clarify the impact of TCR on convergence, although the results can naturally be extended to multi-domain problems.

\subsection{Optimized Schwarz method}
Let $\mathcal{T}_{12}, \mathcal{T}_{21}$ be linear operators defined on the interface $\Gamma$. 
By combining the two equations from \eqref{2.2} linearly with the coefficients 
$(\mathcal{I}, -\mathcal{T}_{12})$ 
and $(-\mathcal{I}, \mathcal{T}_{21})$, 
we can express the transmission conditions in an iterative form,
     \begin{equation}\label{2.3}
     \left\{
     \begin{aligned}
      \bm{n_{1}}\cdot(\kappa_{1}\nabla T_{1}^{n})+\mathcal{T}_{12} T_{1}^{n}&=(-\mathcal{I}+\mathcal{T}_{12}R_{c})\bm{n_{2}}\cdot(\kappa_{2}\nabla T_{2}^{n-1})+\mathcal{T}_{12}T_{2}^{n-1} && \quad\mbox{on }\Gamma,\\
         \bm{n_{2}}\cdot(\kappa_{2}\nabla T_{2}^{n})+\mathcal{T}_{21} T_{2}^{n}&=(-\mathcal{I}+\mathcal{T}_{21}R_{c})\bm{n_{1}}\cdot(\kappa_{1}\nabla T_{1}^{n-1})+\mathcal{T}_{21}T_{1}^{n-1} && \quad\mbox{on }\Gamma,
     \end{aligned}
     \right.
     \end{equation}
where $n$ represents the iteration index. 
The OSM algorithm for the model \eqref{2.1}-\eqref{2.2} involves iteratively solving subproblems on the domains $\Omega_i$ ($i=1,2$), with boundary conditions given by \eqref{2.3} on the interface $\Gamma$, which are updated using the solutions from the adjacent subdomains obtained in the previous iteration. 
Define $h_1^{n-1}$ and $h_2^{n-1}$ as
    \begin{equation}\label{2.6}
    \left \{
    \begin{aligned}
    h_{1}^{n-1}&=(-\mathcal{I}+\mathcal{T}_{12}R_{c})\bm{n_{2}}\cdot(\kappa_{2}\nabla T_{2}^{n-1})+\mathcal{T}_{12}T_{2}^{n-1}
    \quad\mbox{on }\Gamma,\\
    h_{2}^{n-1}&=(-\mathcal{I}+\mathcal{T}_{21}R_{c})\bm{n_{1}}\cdot(\kappa_{1}\nabla T_{1}^{n-1})+\mathcal{T}_{21}T_{1}^{n-1}
     \quad\mbox{on }\Gamma.
     \end{aligned}
    \right.
    \end{equation}
The OSM algorithm then reads: solving until convergence for $n=1,2...$ with given initial data $h_1^0, h_2^0$  
     \begin{equation}\label{2.4}
     \left \{
     \begin{aligned}
     -\nabla\cdot(\kappa_1 \nabla T_{1}^{n}) +c_1 T_{1}^{n}&=f &&\quad\mbox{in }\Omega_{1},\\
     \mathcal{B}T_1^n&=g &&\quad\mbox{on }\partial\Omega_{1}\setminus \Gamma,\\
     \bm{n_{1}}\cdot(\kappa_{1}\nabla T_{1}^{n})+\mathcal{T}_{12} T_{1}^{n}&=h_{1}^{n-1} && \quad\mbox{on }\Gamma,
     \end{aligned}
     \right.
     \end{equation}
and
     \begin{equation}\label{2.5}
     \left \{
     \begin{aligned}
     -\nabla\cdot(\kappa_2 \nabla T_{2}^{n}) +c_2 T_{2}^{n}&=f &&\quad\mbox{in }\Omega_{2},\\
     \mathcal{B} T_2^n&=g &&\quad\mbox{on }\partial\Omega_{2}\setminus \Gamma,\\
     \bm{n_{2}}\cdot(\kappa_{2}\nabla T_{2}^{n})+\mathcal{T}_{21} T_{2}^{n}&=h_{2}^{n-1} && \quad\mbox{on }\Gamma,
     \end{aligned}
     \right.
     \end{equation}
where $h_i^{n-1},i=1,2$ are updated in the following way
    \begin{equation}\label{2.7}
    \left \{
    \begin{aligned}
    h_{1}^{n-1}&=\left(\mathcal{T}_{21}+\mathcal{T}_{12}-
    {\mathcal{T}_{12}R_c}\mathcal{T}_{21}\right)T_2^{n-2}
    +(-\mathcal{I}+{\mathcal{T}_{12}R_c})h_{2}^{n-2}
    \quad\mbox{on }\Gamma,\\
    h_{2}^{n-1}&=\left(\mathcal{T}_{12}+\mathcal{T}_{21}-
     {\mathcal{T}_{21}R_c}\mathcal{T}_{12}\right)T_1^{n-2}
     +(-\mathcal{I}+{\mathcal{T}_{21}R_c})h_{1}^{n-2}
     \quad\mbox{on }\Gamma.
     \end{aligned}
    \right.
    \end{equation} 
Note here that the updating formula \eqref{2.7} doe not require the normal derivative data, hence is easy to implement. 

We comment here that, for simplicity of analysis, we would like to analyze the OSM within a continuous framework to elucidate the relationship between its convergence rate and optimized transmission conditions, with a focus on the presentation of TCR. Similar analysis can also be performed at a discrete level, focusing on the effects of different discretization techniques and the domain's geometry on the performance of the OSM convergence, but it is beyond the scope of the current research.
\subsection{Convergence analysis via energy estimates}
A simple version of the transmission condition \eqref{2.3} can be obtained by setting $\mathcal{T}_{12}=\mathcal{T}_{21}=p \mathcal{I}$, where $p$ is a constant. 
We name it as the \emph{standard Robin} transmission condition since it degenerates to the commonly used Robin condition when $R_c=0$. 
We focus only on the Dirichlet boundary condition $\mathcal{B}=\mathcal{I}$ and the convergence of the OSM algorithm \eqref{2.4}-\eqref{2.7} with the standard Robin transmission condition can be proved using energy estimates.
\begin{theorem}[Convergence of standard Robin]\label{theorem1}
Assume that the constant $p$ is positive and satisfies $p<1/R_c$. For any initial values $h_1^{0}, h_2^{0}$, the numerical solution $T_i^{n}$ generated by the OSM algorithm \eqref{2.4}-\eqref{2.7} using the standard Robin transmission condition converges to the solution $T_i$ of the model problem \eqref{2.1}-\eqref{2.2} in $H^1$ norm. That is to say, it holds for $i=1,2$
$$\|T_i^{n}-T_i\|_{H^1(\Omega_i)}\rightarrow 0, \quad\mbox{ as } n\rightarrow \infty,$$
and the limits of the subdomain solutions, $T_i^\ast:=\lim_{n\rightarrow\infty}T_i^n$, satisfy the interface condition \eqref{2.2}.
\end{theorem}
\begin{proof}
Denote the subdomain errors at each iteration by $e_i^n:=T_i^n-T_i$.
By linearity, the errors $e_i^n, i=1,2$ satisfy
 \begin{equation}\label{2.4;}
     \left \{
     \begin{aligned}
     -\nabla\cdot(\kappa_i \nabla e_{i}^{n}) + c_i e_{i}^{n}
    &=0 &&\quad\mbox{in }\Omega_{i},\\
    e_i^n&=0 &&\quad\mbox{on }\partial\Omega_{i}\setminus \Gamma,\\
     \bm{n}_{i}\cdot(\kappa_{i}\nabla e_{i}^{n})+p e_{i}^{n}&=(-1+pR_{c})\bm{n}_{j}\cdot(\kappa_{j}\nabla e_{j}^{n-1})+pe_{j}^{n-1} && \quad\mbox{on }\Gamma,
     \end{aligned}
     \right.
     \end{equation}
where $j=3-i$.
Define the subspace of $H^1(\Omega_i)$ with vanishing boundary value on $\partial \Omega_i\setminus\Gamma$ as
\begin{equation}\nonumber
\begin{aligned}
H^1_{\Gamma}(\Omega_i):=\{v_i\in H^1(\Omega_i),v_i=0 \mbox{ on }\partial \Omega_i\setminus\Gamma\}.
\end{aligned}
\end{equation}
Testing the governing equation in \eqref{2.4;} by $v_i\in H^1_{\Gamma}(\Omega_i)$, using the boundary condition $e_i^n=0$ on $\partial\Omega_{i}\setminus \Gamma$, one now finds that the subdomain error $e_i^n$ satisfies
\begin{equation}\label{var_gamma}
\begin{aligned}
&a_i(e_i^{n},v_i)=\int_{\Gamma}{\kappa_{i}\bm{n}}_{i}\cdot\nabla e_{i}^{n} v_i ds, \quad \forall v_i\in H_{\Gamma}^1(\Omega_i),
\end{aligned}
\end{equation}
where $a_{i}(e_i^n,v_i)=\int_{\Omega_i}\kappa_i\nabla e_i^n\nabla v_i d\bm{x}+\int_{\Omega_i}c_i e_i^n v_i d\bm{x}$ is the  bilinear form defined on subdomain $\Omega_i$.
We consider first the case where $\min c_i(x)\geq c^0>0$, i.e., the reaction coefficient has a positive lower bound. 
%If $c_i\ge C_0>0$, 
Choosing $v_i=p e_i^{n}$ in formula \eqref{var_gamma}, we get
\begin{equation}\label{int}
\begin{aligned}
\int_{\Gamma}\kappa_i \bm{n}_{i}\cdot\nabla e_{i}^{n} p e_i^{n}ds&=a_i(e_i^n,p e_i^n)\\
& =\int_{\Omega_i}p(\kappa_i|\nabla e_i^{n}|^2+c_i|e_i^{n}|^2)d\bm{x}\\
&\geq p(\kappa^0\| \nabla e_i^{n}\|^2_{L^2(\Omega_i)}+c^0\|e_i^{n}\|^2_{L^2(\Omega_i)})\\
&\ge C\|e_i^{n}\|^2_{H^1(\Omega_i)},
\end{aligned}
\end{equation}
where $C= p\min\{\kappa^0,c^0\}$.
We now define the pseudo-energy as 
\begin{equation}\nonumber
\begin{aligned}
E^{n}:=E_1^{n}+E_2^{n} \mbox{ with } E_i^{n}=\int_{\Gamma}|\kappa_i\bm{n_{i}}\cdot\nabla e_{i}^{n}|^2+|pe_i^{n}|^2ds.
\end{aligned}
\end{equation}
We aim to show that this pseudo-energy is monotonically decreasing. 
A detailed calculation  using the transmission condition in \eqref{2.4;} gives 
\begin{equation}\nonumber
\begin{aligned}
&|\kappa_i\bm{n_{i}}\cdot\nabla e_{i}^{n}|^2+|pe_i^{n}|^2\\
=&(\kappa_i\bm{n_{i}}\cdot\nabla e_{i}^{n}+pe_i^{n})^2-2\kappa_i\bm{n_{i}}\cdot\nabla e_{i}^{n}pe_i^{n}\\
=&\big((-1+pR_{c})\kappa_j\bm{n_{j}}\cdot\nabla e_{j}^{n-1}+pe_{j}^{n-1}\big)^2-2\kappa_i\bm{n_{i}}\cdot\nabla e_{i}^{n}pe_i^{n}\\
=&|\kappa_j\bm{n_{j}}\cdot\nabla e_{j}^{n-1}|^2+|pe_{j}^{n-1}|^2-pR_c(2-pR_c)|\kappa_j\bm{n_{j}}\cdot\nabla e_{j}^{n-1}|^2\\
&-2(1-pR_{c})\kappa_j\bm{n_{j}}\cdot\nabla e_{j}^{n-1}pe_{j}^{n-1}-2\kappa_i\bm{n_{i}}\cdot\nabla e_{i}^{n}pe_i^{n}.
\end{aligned}
\end{equation}
Integrating over $\Gamma$ and summing over the two subdomains for $p<1/R_c$, one finds
\begin{equation}\label{E}
\begin{aligned}
E^{n}&=E^{n-1}-pR_c(2-pR_c)\Sigma_{j=1}^2\int_{\Gamma}|\kappa_j\bm{n_{j}}\cdot\nabla e_{j}^{n-1}|^2ds\\
&\quad -2(1-pR_{c})\Sigma_{j=1}^2\int_{\Gamma}\kappa_j\bm{n_{j}}\cdot\nabla e_{j}^{n-1}pe_{j}^{n-1}ds-2\Sigma_{i=1}^2\int_{\Gamma}\kappa_i\bm{n_{i}}\cdot\nabla e_{i}^{n}pe_i^{n}ds
\\
&\le E^{n-1}-2(1-pR_{c})\Sigma_{j=1}^2\int_{\Gamma}\kappa_j\bm{n_{j}}\cdot\nabla e_{j}^{n-1}pe_{j}^{n-1}ds
\\&\quad -2\Sigma_{i=1}^2\int_{\Gamma}\kappa_i\bm{n_{i}}\cdot\nabla e_{i}^{n}pe_i^{n}ds,
\end{aligned}
\end{equation}
where we have used the assumption $p<1/R_c$. Starting from inequality \eqref{E}, noting the assumption $p<1/R_c$ again and using \eqref{int}, we get
\begin{equation}\label{equation11}
\begin{aligned}
E^{n}\le&E^{n-1}-2(1-pR_{c})C\Sigma_{i=1}^{2}\|e_i^{n-1}\|^2_{H^1(\Omega_i)}-2C\Sigma_{i=1}^{2}\|e_i^{n}\|^2_{H^1(\Omega_i)}\\
\le& E^{n-1}-\hat{C}\Sigma_{i=1}^{2}(\|e_i^{n-1}\|^2_{H^1(\Omega_i)}+\|e_i^{n}\|^2_{H^1(\Omega_i)}),
\end{aligned}
\end{equation}
where $\hat{C}:=2(1-pR_c)C>0$. Inequality \eqref{equation11} shows that the pseudo-energy $E^n$ is decreasing.  Moreover, summing this inequality from $n=1$ to $N$, we find
\begin{equation}\label{inequality13}
E^{N}
+\hat{C}\Sigma_{n=1}^{N}\Sigma_{i=1}^{2}(\|e_i^{n-1}\|^2_{H^1(\Omega_i)}+\|e_i^{n}\|^2_{H^1(\Omega_i)})\le E^{0}.
\end{equation}
Since the initial energy $E^0$ is a constant and all terms on the left-hand side of inequality \eqref{inequality13} are positive, we deduce that $\Sigma_{n=1}^{N}\|e_i^{n}\|^2_{H^1(\Omega_i)}$ is bounded from above for any $N$. Therefore, the infinite series converges, 
\begin{equation*}
\Sigma_{n=1}^{\infty}\|e_i^{n}\|^2_{H^1(\Omega_i)}<\infty,\ i=1,2.
\end{equation*}
This strong statement implies that the general term in this series must go to $0$ as a function of $n$, i.e.,
\begin{equation}\label{err_1}
\|e_i^{n}\|^2_{H^1(\Omega_i)}\rightarrow 0 \mbox{ as } n\rightarrow \infty.
\end{equation}
For the more general case $c_i(x)\ge 0$, equation \eqref{int} becomes
\begin{equation}\nonumber
\begin{aligned}
\int_{\Gamma}\kappa_i \bm{n}_{i}\cdot\nabla e_{i}^{n} p e_i^{n}ds&=\int_{\Omega_i}p(\kappa_i|\nabla e_i^{n}|^2+c_i|e_i^{n}|^2)d\bm{x}\ge p\kappa^0\| \nabla e_i^{n}\|^2_{L^2(\Omega_i)}.
\end{aligned}
\end{equation}
The followed analysis shows that
\begin{equation}\label{nabla_err}
\|\nabla e_i^{n}\|^2_{L^2(\Omega_i)}\rightarrow 0 \mbox{ as } n\rightarrow \infty.
\end{equation}
In addition, it holds
\begin{equation}\label{bound_err}
e_i^n=0 \quad\mbox{on }\partial\Omega_{i}\setminus \Gamma, \quad \mbox{meas}(\partial\Omega_{i}\setminus \Gamma)>0, \ i=1,2.
\end{equation}
Therefore, combining \eqref{nabla_err} and \eqref{bound_err}, and applying the Poincaré-Friedrichs inequality, we deduce that
\begin{equation}\label{err_2}
\|e_i^{n}\|^2_{H^1(\Omega_i)}\le C_1\|\nabla e_i^{n}\|^2_{L^2(\Omega_i)}\rightarrow 0 \mbox{ as } n\rightarrow \infty,
\end{equation}
where $C_1$ is a constant.
Integrating the results in  \eqref{err_1} and \eqref{err_2}, we then obtain 
\begin{equation}\nonumber
\|T_i^{n}-T_i\|^2_{H^1(\Omega_i)}\rightarrow 0 \mbox{ as } n\rightarrow \infty.
\end{equation}

Now we demonstrate that $T_i^\ast=\lim_{n\rightarrow \infty}T_i^n, i=1,2$ satisfy equation \eqref{2.2}. 
Taking limit on both sides of \eqref{2.7}, one then finds
\begin{equation}\nonumber
    \left \{
    \begin{aligned}
     h_{1}^{*}&=\left(\mathcal{T}_{21}+\mathcal{T}_{12}
     {-\mathcal{T}_{12}R_c}\mathcal{T}_{21}\right)
     T_2^{*}{+(-\mathcal{I}+\mathcal{T}_{12}R_c)}h_{2}^{*}
     \quad\mbox{on }\Gamma,\\
     h_{2}^{*}&=\left(\mathcal{T}_{12}+\mathcal{T}_{21}
     {-\mathcal{T}_{21}R_c}\mathcal{T}_{12}\right)
     T_1^{*}{+(-\mathcal{I}+\mathcal{T}_{21}R_c)}h_{1}^{*}
     \quad\mbox{on }\Gamma,
    \end{aligned}
    \right.
    \end{equation}
where $h_i^{\ast}$ is defined via replacing $n-1$ with $\ast$ in equation \eqref{2.6}.
Substituting the converged form of equations \eqref{2.4}--\eqref{2.5} on $\Gamma$ into the above equations, we obtain
    \begin{equation}\label{2.9}
    \left \{
    \begin{aligned}
     {  \mathbf{n_1}\cdot(\kappa_1\nabla T_1^{*}) +
      \mathcal{T}_{12}T_1^{*}=
      (-\mathcal{I}+\mathcal{T}_{12}R_c)
      \mathbf{n_2}\cdot(\kappa_2\nabla T_2^{*}) +
      \mathcal{T}_{12}T_2^{*} \quad\mbox{on }\Gamma,
    }\\
     {  \mathbf{n_2}\cdot(\kappa_2\nabla T_2^{*}) +
      \mathcal{T}_{21}T_2^{*}=
      (-\mathcal{I}+\mathcal{T}_{21}R_c)
      \mathbf{n_1}\cdot(\kappa_1\nabla T_1^{*}) +
      \mathcal{T}_{21}T_1^{*} \quad\mbox{on }\Gamma,
    }
    \end{aligned}
    \right.
    \end{equation}
which implies that the limit $T_i^{*}$ satisfies equation \eqref{2.3}. 
Next we express equation \eqref{2.9} in vector form as follows:
    \begin{equation}\label{2.10}
     \begin{bmatrix}
      \mathcal{I} & \mathcal{T}_{12}\\
         {-\mathcal{I}+ \mathcal{T}_{21}R_c} & \mathcal{T}_{21}
     \end{bmatrix}
     \begin{bmatrix}
      \Phi_1\\ {T_1^{*}}
     \end{bmatrix}
     =\begin{bmatrix}
      {-\mathcal{I}+\mathcal{T}_{12}R_c} & \mathcal{T}_{12}\\
      \mathcal{I} & \mathcal{T}_{21}
     \end{bmatrix}
     \begin{bmatrix}
      \Phi_2\\ {T_2^{*}}
     \end{bmatrix},
    \end{equation}
where $\Phi_1:=\mathbf{n_1}\cdot(\kappa_1\nabla {T_1^{*}})$, $\Phi_2:=\mathbf{n_2}\cdot(\kappa_2\nabla {T_2^{*}})$. Noting that
    \[\begin{bmatrix}
     {-\mathcal{I}+\mathcal{T}_{12}R_c}   & \mathcal{T}_{12}\\
     \mathcal{I} & \mathcal{T}_{21}
    \end{bmatrix}
    =\begin{bmatrix}
     \mathcal{I} & \mathcal{T}_{12}\\
     {-\mathcal{I}+\mathcal{T}_{21}R_c} & \mathcal{T}_{21}\\
    \end{bmatrix}
    \begin{bmatrix}
     -\mathcal{I} & 0 \\
     {R_c\mathcal{I}} & \mathcal{I}
    \end{bmatrix},\]
equation \eqref{2.10} can be rewritten as
    \[\begin{bmatrix}
     \mathcal{I} & \mathcal{T}_{12}\\
     {-\mathcal{I}+\mathcal{T}_{21}R_c} & \mathcal{T}_{21}
    \end{bmatrix}
    \left(
    \begin{bmatrix}
     \Phi_1\\ {T_1^{*}}
    \end{bmatrix}-
    \begin{bmatrix}
     -\mathcal{I} & 0 \\
     {R_c\mathcal{I}} & \mathcal{I}
    \end{bmatrix}
    \begin{bmatrix}
     \Phi_2\\ {T_2^{*}}
    \end{bmatrix}\right)=0.\]
Using the assumption $p<1/R_c$ and the setting $\mathcal{T}_{12}=\mathcal{T}_{21}=p\mathcal{I}$, one finds that the coefficient matrix operator is invertible, we hence conclude that the limit $T_i^{*}$ satisfies the interface condition \eqref{2.2}.  
Therefore, we can assert that the limit $T_i^{*}$
  represents the solution to the model problem defined by equations \eqref{2.1} and \eqref{2.2}.
  \end{proof}
\begin{remark}
It is noteworthy that  the setting
$\mathcal{T}_{12}=\mathcal{T}_{21}=p\mathcal{I}$ of the standard Robin condition is critical for the convergence proof above. The convergence proof for the more general cases where 
$\mathcal{T}_{12}\neq \mathcal{T}_{21}$ via energy estimate remains an open problem.  
However, choosing  $\mathcal{T}_{12}$ and $\mathcal{T}_{21}$ differently will offer great flexibility in the transmission condition design and finally lead to faster convergence.
Along this line, the Fourier analysis can provide more insights, such as the convergence rate that can be optimized to reach an optimal convergence performance.
\end{remark}
\subsection{Convergence analysis via Fourier analysis}\label{sec}
The energy estimates only provide a convergence result for the OSM algorithm \eqref{2.4}-\eqref{2.7} when the standard Robin condition is applied, no convergence rate estimate can be obtained. 
Hereby, we introduce an alternative approach, the Fourier analysis, to step into more insights of the OSM algorithm \eqref{2.4}-\eqref{2.7}. 
To this end, the Fourier transform
\begin{equation}\nonumber
\hat{f}(k)=\mathcal{F}(f):=\int_{-\infty}^{\infty}f(x)e^{-ikx}dx,\quad f(x)=\mathcal{F}^{-1}(\hat{f}):=\frac{1}{2\pi}\int_{-\infty}^{\infty}\hat{f}(k)e^{ikx}dk,
\end{equation}
will be applied, where $k\in\mathbb{R}$ denotes the Fourier frequency. 
Thus, we consider the two-dimensional problem and assume that the defining domain is infinite, 
\begin{equation}\label{infinitedomain}
\Omega=\Omega_1\cup\Omega_2\quad \mbox{with}\quad \Omega_1:=(-\infty,0)\times \mathbb{R},\ \Omega_2:=(0,\infty)\times \mathbb{R},
\end{equation}
such that the Fourier transform can be applied directly.
In addition, we assume also that both $k_i(x)$ and $c_i(x)$ are constant functions independent of the variable $x$, so that the differential operators can be explicitly described by their Fourier symbols.
By linearity, it suffices to consider the case $f= 0$ and analyze the convergence to zero. Applying a Fourier transform in the $x_2$-direction to both equations \eqref{2.4} and \eqref{2.5} (noting that the outer boundary operators $\mathcal{B}_i$ are absent due to the infinite domain assumption) results in the following transformed system:
    \begin{flalign}
    &\left \{
    \begin{aligned}\label{3.1}
    &(-\kappa_{1}\partial_{x_1 x_1}+c_{1}+\kappa_{1}k^2)\hat{T}^{n}_{1}=0  &&\quad x_1\in(-\infty, 0),\\
    &(\kappa_{1} \partial_{x_1}+\tau_{12}) \hat{T}^{n}_{1}=((1-\tau_{12}R_{c})\kappa_{2}\partial_{x_1}+\tau_{12})\hat{T}^{n-1}_{2}  && \quad x_1=0,
    \end{aligned}
    \right.\\
    &\left\{
    \begin{aligned}\label{3.2}
    &(-\kappa_{2}\partial_{x_1 x_1}+c_{2}+\kappa_{2}k^2)\hat{T}^{n}_{2}=0  &&\quad x_1\in(0,+\infty),\\
    &(-\kappa_{2} \partial_{x_1}+\tau_{21} )\hat{T}^{n}_{2}=((-1+\tau_{21}R_{c})\kappa_{1}\partial_{x_1}+\tau_{21})\hat{T}^{n-1}_{1}  && \quad x_1=0,
    \end{aligned}
    \right.
    \end{flalign}
where we have used the simple fact $\bm{n_1}=(1,0),\bm{n_2}=(-1,0)$ to simplify the notation, $k$ denotes the Fourier frequency, and $\tau_{12},\tau_{21}$ are the Fourier symbols of the linear operators $\mathcal{T}_{12},\mathcal{T}_{21}$, respectively. 
Solving the first equations in \eqref{3.1} and \eqref{3.2}, we obtain by the property that the solution should be bounded at infinity
\begin{equation}\label{3.3}
    \hat{T}_{1}^{n}=A_1^{n}e^{\lambda_{1}(k)x_1},\quad\hat{T}_{2}^{n}=B_2^{n}e^{-\lambda_{2}(k)x_1},
    \end{equation}
where 
    \begin{equation}\label{characteristicroot}
    \lambda_{i}(k)=\frac{\sqrt{c_i\kappa_i+\kappa^2_ik^2}}{\kappa_i}.
    \end{equation}
Substituting the solutions \eqref{3.3} into the transformed transmission conditions in \eqref{3.1} and \eqref{3.2}, and iterating between the subdomains, we derive
    \begin{equation}\nonumber 
    A_1^{n}=\rho A_1^{n-2},\quad
    B_2^{n}=\rho B_2^{n-2},          
    \end{equation}
where
    \begin{equation}\label{3.5}
    \rho(k,\tau_{12},\tau_{21})=\frac{(1-\tau_{12}R_{c})-\tau_{12}\xi_{2}(k)}
    {1+\tau_{12}\xi_{1}(k)}\cdot \frac{(1-\tau_{21}R_{c})-\tau_{21}\xi_{1}(k)}
    {1+\tau_{21}\xi_{2}(k)}
    \end{equation}
is the convergence factor with 
\begin{equation}\label{f_k}
\xi_1(k)=\frac{1}{\kappa_1\lambda_1(k)},\quad \xi_2(k)=\frac{1}{\kappa_2\lambda_2(k)}.
\end{equation}
Noting that the convergence factor $\rho$ is symmetric in the Fourier frequency $k$, we only need to consider the positive frequencies $k\in[k_{min},k_{max}]$ involved in the numerical discretization, with $k_{min}$ being the smallest and $k_{max}$ the largest positive frequencies.

Now, the convergence of the OSM algorithm with standard Robin transmission condition (i.e., $\mathcal{T}_{12}=\mathcal{T}_{21}=p\mathcal{I}$ so that $\tau_{12}=\tau_{21}=p$) can be analyzed below. 
\begin{theorem}[Convergence of standard Robin]\label{theorem2}
Assume that $R_c,\kappa_i>0$, $c_i\ge0$, $i=1,2$ are constants. Then, for all positive $p<2/R_c$,  the convergence factor corresponding to the setting $\tau_{12}=\tau_{21}=p$ satisfies
    \begin{equation}\label{rho<1} 
    \left|\rho(k,p,p) \right|= \left|\frac{1-p(R_{c}+\xi_{2}(k))}
    {1+p\xi_{1}(k)}\cdot \frac{1-p(R_{c}+\xi_{1}(k))}
    {1+p\xi_{2}(k)}\right|<1 ,\quad \forall k\in[k_{min},k_{max}].
    \end{equation}
In other words, when applied to solving the model problem \eqref{2.1}-\eqref{2.2} with domain setting \eqref{infinitedomain}, the OSM algorithm \eqref{2.4}-\eqref{2.7} with the standard Robin condition converges for all positive $p<{2}/{R_c}$.
\end{theorem}
\begin{proof}
Detailed calculation shows that, to prove the inequality \eqref{rho<1}, it is sufficient to show   
    \begin{equation}\label{3.13}
    \begin{aligned}
    & p(pR_c-2)(R_c+\xi_1(k)+\xi_2(k))<0,\\
    &(p R_c-1)^2+1+(\xi_1(k)+\xi_2(k))R_c+2\xi_1(k)\xi_2(k)p^2>0.
    \end{aligned}
    \end{equation}
Using the fact that both $\xi_1(k)$ and $\xi_2(k)$ are positive (see also Lemma \ref{lemma3.1}), one finds that the second inequality in \eqref{3.13} holds naturally and the first inequality is ensured by the condition $0<p<2/R_c$. 

\end{proof}
\begin{remark}
The Fourier analysis in Theorem \ref{theorem2} shows a wider range of positive $p$ values for which the OSM algorithm \eqref{2.4}-\eqref{2.7} converges, compared to the energy estimate in Theorem \ref{theorem1}. However, both analyses indicate that the Robin parameter $p$ must be bounded from above by a factor of $1/R_c$. This means that for solving problems with large TCR, the Robin parameter 
$p$ should be chosen to be small, although the optimal choice is still unclear.
\end{remark}
\begin{remark}
The Fourier analysis directly applies to 3-dimensional problems. Consider the defining domain $\Omega=\mathbb{R}^3=\Omega_1\cup \Omega_2$ where $\Omega_1:=(-\infty,0)\times(-\infty,\infty)^2$ and $\Omega_2:=(0,+\infty)\times(-\infty,\infty)^2$. One must then perform the Fourier transform in both the $x_2$ and $x_3$ directions, with $k_2$ and $k_3$ denoting the corresponding Fourier frequencies. The characteristic roots are then given by $ \lambda_{i}(k_2,k_3)=\frac{\sqrt{c_i\kappa_i+\kappa^2_i(k_2^2+k_3^2)}}{\kappa_i}$, which is the same as the $\lambda_i(k)$ given in equation \eqref{characteristicroot}, except that $k^2$ is replaced by $k_2^2+k_3^2$. 
Therefore, the convergence factor retains the same form as in the 2-dimensional case, and the convergence result in Theorem \ref{theorem2} still holds.
\end{remark}
Given the explicit formula of the convergence factor $\rho(k,p,p)$ for the OSM algorithm \eqref{2.4}-\eqref{2.7} with the standard Robin transmission condition, actually, one can choose the optimal parameter $p$ by optimizing the convergence factor $\rho(k,p,p)$. However, this task is not straightforward, as the convergence factor $\rho(k,p,p)$ is too complicated to be explicitly analyzed. Additionally, the convergence factor $\rho$ in the Fourier frequency domain provides significant flexibility for analyzing more general transmission conditions. Therefore, in the next section, we will design and optimize a more efficient transmission condition. 

\section{Scaled Robin transmission condition}\label{sec3}
Setting the Fourier symbols $\tau_{12}$ and $\tau_{21}$ as
\begin{equation}\nonumber
\tau_{12}^{opt}(k):=\frac{1}{\xi_{2}(k)+R_{c}}, \quad \tau_{21}^{opt}(k):=\frac{1}{\xi_{1}(k)+R_{c}}
    \end{equation}
results in the convergence factor  $\rho$ in \eqref{3.5} being identically zero, which indicates that the corresponding OSM algorithm converges in just two iterations.
However, performing an inverse Fourier transform on the symbols  $\tau_{12}^{opt}(k)$ and $\tau_{21}^{opt}(k)$ yields non-local operators $\mathcal{T}_{12}^{opt}$ and $\mathcal{T}_{21}^{opt}$, which are computationally expensive in practice.
Consequently, our objective is to seek local approximations of the optimal Fourier symbols to ensure that the algorithm remains straightforward to implement while achieving rapid convergence. 
Choosing $\tau_{12},\tau_{21}$ in a structurally consistent manner as follows
\begin{equation}\label{3.16}
    \tau_{12}=\frac{1}{\xi_{2}(p)+R_{c}},\quad \tau_{21}=\frac{1}{\xi_{1}(p)+R_{c}}
    \end{equation}
greatly simplifies the optimization for predicting the best transmission parameter, while also reflecting the heterogeneity between layered materials, where $p$ is a constant to be determined for fast convergence. Since it scales at each subdomain using the subdomain parameters, we refer to this transmission condition as the \emph{scaled Robin}.
We note that, for the frequency $k=p$, the scaled Robin transmission condition yields an exact solver that converges within two iterations.
By inserting \eqref{3.16} into equation \eqref{3.5}, one finds the convergence factor corresponding to the scaled Robin transmission condition:
\begin{equation}\nonumber
    \rho_{SR}(k,p):=\frac{\xi_{1}(p)-\xi_{1}(k)}
    {\xi_{1}(p)+R_{c}+\xi_{2}(k)}\cdot \frac{\xi_{2}(p)-\xi_{2}(k)}
    {\xi_{2}(p)+R_{c}+\xi_{1}(k)}.
    \end{equation} 
To  identify the optimal transmission parameter $p^\ast$ for the associated OSM algorithm, and noting the fact that $\xi_i(p)=\xi_i(-p)$, we only need to address the following optimization problem:
    \begin{equation}\label{3.18}
    \underset{p>0}{\min}
    \underset{k\in [k_{min},k_{max}]}{\max}\left|\rho_{SR}(k,p)\right|,
    \end{equation}
where $k_{min}$, the smallest frequency involved, is estimated as $\pi/|\Gamma|$ if Dirichlet boundary conditions are applied, while $k_{max}$, the largest frequency, is estimated as $\pi/h$ when a uniform mesh of size $h$ is used \cite{strauss2007partial}.

The following lemma is essential for analyzing the optimization problem \eqref{3.18}.

\begin{lemma}\label{lemma3.1}
Assume that $\kappa_i > 0$ and $c_i \ge 0$ for $i = 1, 2$. The functions $\xi_i(k),i=1,2$ are positive and monotonically decrease in $k$ for $k>0$.
\end{lemma}

\begin{proof}
From equation \eqref{f_k}, it follows that $\xi_i(k) > 0$ for $i = 1, 2$.
Taking the derivative of $\xi_i(k)$ with respect to $k$, we obtain
    \begin{equation}\nonumber 
     \frac{\partial \xi_i(k)}{\partial k}=-\frac{k\kappa_i^2}{(k^2\kappa_i^2+c_i\kappa_i)^{3/2}}<0,\quad i=1,2.
    \end{equation}
This indicates that $\xi_i(k)$, for $i = 1, 2$, are strictly monotonically decreasing in $k$ for $k > 0$.
\end{proof}

Now we are ready to show the following basic result on the convergence of the OSM algorithm \eqref{2.4}-\eqref{2.7} with scaled Robin transmission condition. 
\begin{theorem}[Convergence of scaled Robin]\label{theorem3.12}
Assume that $R_c,\kappa_i>0$ and $c_i\ge0$ for $i=1,2$. For any real $ p$, the convergence factor $\rho_{SR}$ satisfies
\begin{equation}\nonumber 
       \left|\rho_{SR}(k,p) \right|= \left|\frac{\xi_{1}(p)-\xi_{1}(k)}
      {\xi_{1}(p)+R_{c}+\xi_{2}(k)}\cdot \frac{\xi_{2}(p)-\xi_{2}(k)}
      {\xi_{2}(p)+R_{c}+\xi_{1}(k)}\right|<1 ,\quad  \forall k\in[k_{min},k_{max}].
    \end{equation}
\end{theorem}

\begin{proof}
From the monotonicity of $\xi_i(k)$ indicating by Lemma \ref{lemma3.1}, both terms in $\rho_{SR}$ share the same sign, leading to $\left|\rho_{SR}(k,p) \right|=\rho_{SR}(k,p)$. 
To establish convergence, we need to demonstrate
    \begin{equation}\nonumber
     \rho_{SR}(k,p)=\frac{\xi_{1}(p)-\xi_{1}(k)}
     {\xi_{1}(p)+R_{c}+\xi_{2}(k)}\cdot \frac{\xi_{2}(p)-\xi_{2}(k)}
     {\xi_{2}(p)+R_{c}+\xi_{1}(k)}<1.
    \end{equation}
This expression can be reformulated as
    \begin{equation}\nonumber
     (\xi_{1}(p)-\xi_{1}(k))(\xi_{2}(p)-\xi_{2}(k))<(\xi_{1}(p)+R_{c}+\xi_{2}(k))(\xi_{2}(p)+R_{c}+\xi_{1}(k)).
    \end{equation}
Upon expanding both sides, it becomes apparent that this inequality holds if and only if $(\xi_{1}(p)+R_{c}+\xi_{2}(p))(\xi_{1}(k)+R_{c}+\xi_{2}(k))>0$. This condition is evidently satisfied due to the nonnegativity of $\xi_i$.
\end{proof}

The first step for solving the optimization problem \eqref{3.18} involves constraining the search range for the optimal transmission parameter, as shown below.

\begin{lemma}[Restriction for the interval of $p$]\label{lemma3.13}
If $p^*$ is a solution to the min-max problem \eqref{3.18}, then $p^*$ must reside within the interval $[k_{min},k_{max}]$. Thus, the following min-max problems are equivalent:
    \begin{equation}\nonumber 
    \underset{p>0}{\min}
    \underset{k\in [k_{min},k_{max}]}{\max}\left|\rho_{SR}(k,p)\right|\Longleftrightarrow\underset{p\in[k_{min},k_{max}]}{\min}
    \underset{k\in [k_{min},k_{max}]}{\max}\left|\rho_{SR}(k,p)\right|.
    \end{equation}
\end{lemma}

\begin{proof}
We directly analyze the convergence factor $\rho_{SR}(k,p)$, as it is guaranteed to be positive. 
Calculating the derivative of   $\rho_{SR}$ with respect to $p$, we obtain 
    \begin{equation}\label{3.19}
    \frac{\partial \rho_{SR}}{\partial{p}}=\frac{g_p(\xi_1(k)+\xi_2(k)+R_c)}{(\xi_1(p)+\xi_2(k)+R_c)^2(\xi_2(p)+\xi_1(k)+R_c)^2},
    \end{equation}
where
 $g_p=\frac{d\xi_1(p)}{dp}(\xi_2(p)-\xi_2(k))(\xi_2(p)+\xi_1(k)+R_c)+\frac{d\xi_2(p)}{dp}(\xi_1(p)-\xi_1(k))(\xi_1(p)+\xi_2(k)+R_c)$. 
 Utilizing the fact that $\xi_i(p),i=1,2$, are monotonically decreasing functions of $p$ (cf. Lemma \ref{lemma3.1}), we observe that, for all $k \in [k_{min}, k_{max}]$, the derivative  $\frac{\partial \rho_{SR}}{\partial p} (k,p)< 0$ when $p<k_{min}$. 
This indicates that we are not at an optimum since increasing $p$ would lead to a decrease in the convergence factor for all frequencies $k \in [k_{min},  k_{max}]$.

A similar argument holds for $p>k_{max}$, thereby concluding the proof. 
\end{proof}

Next, we investigate the local maxima of the convergence factor $\rho_{SR}$ with respect to $k$.

\begin{lemma}[Local maxima in $k$]\label{lemma3.14}
For any fixed $p\in [k_{min},k_{max}]$, the convergence factor $\rho_{SR}(k,p)$ attains its maximum exclusively at the endpoints $k_{min}$ and $k_{max}$, i.e.,
    \begin{equation}\nonumber
    \underset{k\in [k_{min},k_{max}]}{\max}\rho_{SR}(k,p)=\max \{\rho_{SR}(k_{min},p),\rho_{SR}(k_{max},p)\}.
    \end{equation}
\end{lemma}

\begin{proof}
Differentiating  $\rho_{SR}(k, p)$ with respect to $k$ gives
    \begin{equation}\nonumber
     \frac{\partial \rho_{SR}}{\partial{k}}=
     \frac{g_k(\xi_1(p)+\xi_2(p)+R_c)}{(\xi_1(p)+\xi_2(k)+R_c)^2(\xi_2(p)+\xi_1(k)+R_c)^2},
    \end{equation}
where
$g_k=\frac{d\xi_1(k)}{dk}(\xi_2(k)-\xi_2(p))(\xi_1(p)+\xi_2(k)+R_c)+\frac{d\xi_2(k)}{dk}(\xi_1(k)-\xi_1(p))(\xi_1(k)+\xi_2(p)+R_c)$. 
If $p = k_{min}$, then $ \partial_k \rho_{SR}>0$,  indicating that the convergence factor $\rho_{SR}$ is strictly increasing in $k$. 
Thus, the maximum occurs at $k=k_{max}$, giving us $\max \rho_{SR}=\rho_{SR}(k_{max},p)$. If $p=k_{max}$, a similar argument applies, showing that $\max \rho_{SR}=\rho_{SR}(k_{min},p)$.
If  $p\in(k_{min},k_{max})$, then $\partial_k \rho_{SR}<0$ for $k\in(k_{min},p)$ and $\partial_k \rho_{SR}>0$, for $k>p$. Consequently, $\rho_{SR}$ decreases in $(k_{min},p)$ and increases in $(p, k_{max})$. Therefore, the maximum must occur at one of the endpoints, $k=k_{min}$ or $k=k_{max}$. This concludes the proof. 
\end{proof} 

Now we are ready to solve the min-max problem \eqref{3.18}.

\begin{theorem}[Optimization of the scaled Robin condition]\label{theorem3.15} The min-max problem \eqref{3.18} is solved by the unique positive root $p^\ast$ of the nonlinear equation
    \begin{equation}\label{3.20}
    \rho_{SR}(k_{min},p^*)=\rho_{SR}(k_{max},p^*).
    \end{equation}
\end{theorem}

\begin{proof}
According to Lemma \ref{lemma3.14}, the min-max problem \eqref{3.18} is equivalent to
\begin{equation}\nonumber
     \underset{p\in [k_{min},k_{max}]}{\min}
    {\max}\{\rho_{SR}(k_{min},p),\rho_{SR}(k_{min},p)\}.
    \end{equation}
From \eqref{3.19}, we have the following observations: (1) The derivative $\partial_p \rho_{SR}(k_{min},\linebreak[3]p)>0$ for any $p\in(k_{min},k_{max}]$.
This indicates that increasing $p\in[k_{min},k_{max}]$ strictly increases $\rho_{SR}(k_{min},p)$ from $\rho_{SR}(k_{min},k_{min})=0$ to $\rho_{SR}(k_{min},k_{max})>0$. 
(2) In contrast, $\partial_p \rho_{SR}(k_{max},p)<0$ for any $ p\in[k_{min},k_{max})$. Consequently, increasing $p$ from $k_{min}$ to $k_{max}$ strictly decreases $\rho_{SR}(k_{max},p)$ from $\rho_{SR}(k_{max},\linebreak[3]k_{min})>0$ to $\rho_{SR}(k_{max},k_{max})=0$.
By the continuity of the convergence factor, there exists a unique $p^{*}\in(k_{min},k_{max})$
that satisfies $\rho_{SR}(k_{min},p^*)=\rho_{SR}(k_{max},p^*)$ and thus solves the min-max problem \eqref{3.18}. 
\end{proof}

Due to the complex nature of the convergence factor, directly solving equation \eqref{3.20} for the optimized scaled Robin parameter $p^\ast$ when $c_{1,2}\neq0$ proves to be quite challenging. Fortunately, several asymptotic results can be derived, as demonstrated below.

\begin{theorem}[Scaled Robin asymptotics]\label{theorem3.16}
Let $R_c,\kappa_i>0$ and $c_i\ge0$ for $i=1,2$, and $k_{max}=\frac{\pi}{h}$. The solution $p^\ast$ to the min-max problem \eqref{3.18} has the following asymptotic expansion for small $h$
    \begin{equation}\label{3.21}
    p^\ast= C_p+O(h),
    \end{equation} 
where $C_p$ is a constant determined by the nonlinear equation 
    \begin{equation}\label{3.22}
    G_{k_{min}}(C_p)=G_{k_{max}}(C_p),
    \end{equation}
with $G_{k_{min}},G_{k_{max}}$ defined in \eqref{3.23}.
Furthermore, the convergence factor satisfies the following estimate
    \begin{equation}\label{convest} 
    \underset{k_{min}\le k \le k_{max}}{\max}\rho_{SR}(k,p^*)= G_{k_{max}}(C_p)+O(h),
    \end{equation}
where $G_{k_{max}}(C_p)<1$ is a positive constant  independent of the meshsize $h$.
\end{theorem}

\begin{proof}
Guided by numerical investigations, we make the ansatz $p^{*}=C_p+D_ph^{\alpha},\alpha>0$ and plug it into  \eqref{3.20}. 
Then, expanding the convergence factor for $h$ small gives
    \begin{equation}\nonumber
      \rho_{SR}(k_{min},p^{*})=G_{k_{min}}(C_p)+O(h),\quad
      \rho_{SR}(k_{max},p^{*})= G_{k_{max}}(C_p)+O(h),
    \end{equation} 
where $G_{k_{min}}(C_p),G_{k_{max}}(C_p)$ read
    \begin{equation}\label{3.23}
    \begin{aligned}
      &G_{k_{min}}(C_p)=\frac{\kappa_1\kappa_2(\hat{\varphi}_1(k_{min})-\hat{\varphi}_1(C_p))(\hat{\varphi}_2(k_{min})-\hat{\varphi}_2(C_p))}{\hat{\varphi}_3\hat{\varphi}_4},\\
      &G_{k_{max}}(C_p)=\frac{1}{(R_c\kappa_1\hat{\varphi}_1(C_p)+1)(R_c\kappa_2\hat{\varphi}_2(C_p)+1)},
      \end{aligned}
    \end{equation}
with
$\hat{\varphi}_{1}(x)=\sqrt{x^2+\eta_1}$, $\hat{\varphi}_{2}(x)=\sqrt{x^2+\eta_2}$,
$\hat{\varphi}_3=\kappa_1(R_c\kappa_2\hat{\varphi}_2(C_p)+1)\hat{\varphi}_1(k_{min})+\kappa_2\hat{\varphi}_2(C_p)$, 
$\hat{\varphi}_4=\kappa_2(R_c\kappa_1\hat{\varphi}_1(C_p)+1)\hat{\varphi}_2(k_{min})+\kappa_1\hat{\varphi}_1(C_p)$.
Setting the leading terms equal leads to \eqref{3.22}, where  the root $C_p$ is a constant independent of $h$ since both $G_{k_{min}}(C_p)$ and $G_{k_{max}}(C_p)$ are $h$ independent. 
Additionally, the estimate for the convergence factor in \eqref{convest} follows from Theorem \ref{theorem3.15}, and the second equation in \eqref{3.23} confirms that $G_{k_{max}}(C_p)<1$.  
\end{proof}

The implicit nature of the optimized parameter $p^\ast$ defined in \eqref{3.21} and the estimate $G_{k_{max}}(C_p)$ say nothing about how the physical parameters influence the performance of the OSM algorithm. To make it more clear, we simplify our model problem to the case where $c_i=0, i=1,2$, which allows for the derivation of several detailed results. 

\begin{theorem}[Scaled Robin asymptotics for $c_i=0$]\label{theorem4.6}
Assume $R_c,\kappa_i>0$ and $k_{max}=\frac{\pi}{h}$. For the special case $c_i=0,i=1,2$, the solution $p^\ast$ to the min-max problem \eqref{3.18} is given in closed form by 
\begin{equation}\label{p_01}
 p^\ast= \frac{\pi R_c\kappa_1\kappa_2k_{min}+\sqrt{\pi}\sqrt{((\pi R_c\kappa_1+h)\kappa_2+h\kappa_1)k_{min}((R_c\kappa_1k_{min}+1)\kappa_2+\kappa_1)}}{(R_c(hk_{min}+\pi)\kappa_2+h)\kappa_1+h\kappa_2},
\end{equation}
which has the following asymptotic expansion for small $h$
    \begin{equation}\label{p_0} 
    p^\ast=\hat{C}_p+O(h),\quad \hat{C}_p= k_{min}+\sqrt{k_{min}^2+\frac{1}{R_c\kappa_1}+\frac{1}{R_c\kappa_2}}.
    \end{equation}
    The corresponding convergence factor estimate reads
    \begin{equation}\label{e4.11}
\underset{k_{min}\le k \le \frac{\pi}{h}}{\max}|\rho_{SR}(k,p^*)| = \frac{1}{(R_c\kappa_1\hat{C}_p+1)(R_c\kappa_2\hat{C}_p+1)}+O(h).
    \end{equation} 
\end{theorem}

\begin{proof}
For the special case $c_i=0,i=1,2$ the solution $p^\ast$ to the equi-oscillation problem \eqref{3.20}  can be explicitly derived, yielding equation \eqref{p_01}. The result in equation \eqref{p_0} follows from a Taylor expansion, while equation \eqref{e4.11} is a consequence of Theorem \ref{theorem3.16}.
\end{proof}

\begin{remark}
The explicit expression of $p^*$ in \eqref{p_0} offers a robust initial value for determining the optimized parameter $p^\ast$ for the case $c_i>0, i=1,2$ from \eqref{3.20}.
\end{remark}

\begin{remark}\label{remark4.8}
In contrast to the findings presented in \cite{gander2015optimized} and \cite{gander2019heterogeneous}, introducing TCR $R_c$ brings a variety of new properties into the OSM algorithm. 

First of all, when addressing the heat conduction in composites \eqref{2.1}-\eqref{2.2} inclusive of TCR $R_c$, the scaled Robin condition facilitates mesh-independent convergence behavior, as illustrated in equation \eqref{e4.11}. As a comparison, the convergence is mesh-dependent when no TCR is under consideration, as shown in \cite{gander2015optimized} and \cite{gander2019heterogeneous}.

We now look at the heterogeneity contrast defined by
$\lambda=\frac{\kappa_1}{\kappa_2}$.
By symmetry, we restrict our analysis to the case where $\kappa_1\geq\kappa_2$, i.e., $\lambda\geq 1$. Reformulating the estimate from equation \eqref{e4.11} yields
\begin{equation}\nonumber
    \underset{k_{min}\le k \le \frac{\pi}{h}}{\max}|\rho_{SR}(k,p^*)| = \frac{1}{\lambda T(\lambda)}\cdot\frac{1}{T(\lambda)+1} +O(h),
    \end{equation}
where 
$  
T(\lambda)=R_c\kappa_2 k_{min}+\sqrt{R_c \kappa_2 k_{min}(R_c\kappa_2 k_{min}+1+1/\lambda)}.
$
A further asymptotic expansion for $\lambda$ large yields
\begin{equation}\nonumber
  \begin{aligned}
   \underset{k_{min}\le k \le \frac{\pi}{h}}{\max}|\rho_{SR}(k,p^*)|= \frac{1}{T_\infty(T_\infty+1)}\cdot\frac{1}{\lambda}+O(\frac{1}{\lambda^2})+O(h),
    \end{aligned}
    \end{equation}
where $T_\infty=\lim_{\lambda\rightarrow \infty} T(\lambda)=R_c\kappa_2k_{min}+\sqrt{R_c\kappa_2k_{min}(R_c\kappa_2k_{min}+1)}$.
Consequently, similar to the conclusions drawn in \cite{gander2015optimized} and \cite{gander2019heterogeneous}, it is evident that an increased heterogeneity contrast correlates with expedited convergence of the OSM algorithm. 
However, unlike scenarios devoid of TCR, the estimate for the convergence factor is influenced not only by the heterogeneity contrast but also by the thermal conductivity. 
It can be concluded that, for fixed $\lambda$, a larger thermal conductivity would lead to faster convergence of the OSMs. To sum up, for fixed $\kappa_2$, the heterogeneity contrast can improve the convergence, while for fixed heterogeneity contrast $\lambda$, the thermal conductivity also enhances convergence. 
\end{remark}

\begin{remark}\label{scaled_rc}
A significant TCR $R_c$ facilitates a faster convergence of the OSM algorithm, as initially indicated by the convergence factor. 
To elucidate this result more clearly, we examine the following two cases:
\begin{itemize}
    \item Asymptotically expanding $\max_k|\rho_{SR}(k,p^*)|$ in $R_c$ for $R_c$ large gives
    \begin{equation}\nonumber
    \underset{k_{min}\le k \le \frac{\pi}{h}}{\max}|\rho_{SR}(k,p^*)| \sim \frac{1}{4\kappa_1\kappa_2k_{min}^2R_c^2}+O(\frac{1}{R_c^3}),
    \end{equation}
    which shows that large TCR accelerates the convergence quadratically. 
    \item Taylor expanding $\max_k|\rho_{SR}(k,p^*)|$ in $R_c$ gives
    \begin{equation}\nonumber 
    \begin{aligned}
   \underset{k_{min}\le k \le \frac{\pi}{h}}{\max}|\rho_{SR}(k,p^*)| \sim 1-(\kappa_1+\kappa_2)\sqrt{\frac{k_{min}}{\kappa_2}}R_c^{\frac{1}{2}}+O(R_c),
    \end{aligned}
    \end{equation}
    which clearly shows that it converges to the case without TCR at a rate of the square root of $R_c$.   
\end{itemize} 
Furthermore, the results above also illustrate the role of the thermal conductivity $\kappa_i$ in enhancing convergence.
\end{remark}

To well compare the results with those without TCR, we present below the asymptotic results for $R_c = 0$, which are similar to the results reported in \cite{gander2019heterogeneous} for $c_1>0, c_2=0$. 

\begin{theorem}[Scaled Robin asymptotics for $R_c=0$]\label{theorem3.18}
Assume  $\kappa_i>0,c_i\ge0,i=1,2$, and $k_{max}=\frac{\pi}{h}$. 
For the case neglecting TCR, i.e., $R_c=0$, the optimized scaled Robin parameter  $p_0^*$ and the corresponding convergence rate estimate satisfy for $h$ small
   \begin{equation}\label{p_rc0}
   \begin{aligned}
   p_0^*&\sim\frac{\sqrt{\pi(\lambda+1)(\lambda \hat{\varphi}_1(k_{min})+\hat{\varphi}_2(k_{min}))}}{\lambda+1}h^{-\frac{1}{2}},\\
    \underset{k_{min}\le k \le \frac{\pi}{h}}{\max}|\rho_{SR0}(k,p_0^*)| &\sim 1-\frac{(\lambda+1)^{\frac{3}{2}}\sqrt{(\lambda \hat{\varphi}_1(k_{min}))+\hat{\varphi}_2(k_{min}))}}{\sqrt{\pi}\lambda}h^{\frac{1}{2}}.
    \end{aligned}
    \end{equation}
\end{theorem}

\begin{proof}
Firstly, the convergence factor for $R_c=0$ reads
     \begin{equation}\nonumber 
     \rho_{SR0}(p,k)=\frac{(\sqrt{p^2+\eta_1}-\sqrt{k^2+\eta_1})(\sqrt{p^2+\eta_2}-\sqrt{k^2+\eta_2})\kappa_1\kappa_2}{(\kappa_1\sqrt{p^2+\eta_1}+\kappa_2\sqrt{k^2+\eta_2})(\kappa_1\sqrt{k^2+\eta_1}+\kappa_2\sqrt{p^2+\eta_2})}.
     \end{equation}
An analysis analogous to that used in \cite{gander2019heterogeneous} indicates that the min-max problem $$\underset{p\in R}{\min}\underset{k\in[k_{min},k_{max}]}{\max}\rho_{SR0}(p,k)$$ is solved by  the equioscillation at endpoints $k_{min}$ and $k_{max}$. 
Making the ansatz $p_0^*=C_{r}h^{-\alpha_1},\alpha_1>0$ and Taylor expanding the convergence factor at $h = 0$, we get
    \begin{equation}\nonumber
    \begin{aligned}
    &|\rho_{SR0}(k_{min},p_0^*)|\sim1-\frac{(\lambda+1)(\lambda \hat{\varphi}_1(k_{min})+\hat{\varphi}_2(k_{min}))}{\lambda C_{r}}h^{\alpha_1},\\
    &|\rho_{SR0}(k_{max},p_0^*)|\sim1-\frac{C_{r}(\lambda+1)^2}{\pi \lambda}h^{1-\alpha_1}.
    \end{aligned}
    \end{equation}
Identifying the previous two expansions leads to the desired result.
\end{proof}

Note that this result degenerates to the case in \cite{gander2019heterogeneous} if setting $\eta_2=\frac{c_2}{\kappa_2}=0$, to the case in \cite{gander2015optimized} if setting $\eta_i=\frac{c_i}{\kappa_i}=0,i=1,2$.

\begin{remark}[Comparison of the convergence factors]
We compare the convergence factor for the scaled Robin condition with the standard Robin. Since lack of the optimization result for the standard Robin condition, the Matlab command {\textbf{fminsearch}} is used to numerically get an optimized result, while for the scaled Robin we directly use the result indicated in Theorem \ref{theorem4.6}. 
In Fig. \ref{fig4.1}, we plot both optimized convergence factors as functions of the frequency $k$ for the case $c_i=0,i=1,2$ and a fixed thermal conductivity $\kappa_2=0.01$.
Firstly, it is evident that the two convergence factors exhibit distinct behaviors. 
The scaled Robin condition leads to a nonnegative convergence factor equi-oscillating at $k_{min}$ and $k_{max}$. 
In contrast, the convergence factor for the standard Robin condition can take negative values, forming a local interior maximum point and eventually leading to more complex equi-oscillation regime and complicating the theoretical analysis, this is the reason why the standard Robin is hard to optimize.
Secondly, the scaled Robin condition generally exhibits a smaller convergence factor compared to the standard Robin condition, particularly when there is a significant heterogeneity contrast or a larger thermal conductivity
$\kappa_1$. 
Thirdly, when the TCR 
$R_c$ is close to zero, the maxima of both convergence factors approach to $1$, indicating slow convergence. 
In contrast, the introduction of TCR results both convergence factors in relatively smaller maxima, which indicates improved convergence speeds.
Fourthly, we find that when the TCR is present, both convergence factors do not change over high frequencies, meaning that both Robin conditions lead to mesh-independent convergence behavior. 
Finally, by comparing the two plots, we observe that for both Robin conditions, a larger 
$\kappa_1$ 
leads to smaller convergence factors, confirming that both large heterogeneity contrast and large conductivity would accelerate the convergence.
The above findings not only confirm the validity of our theoretical results for the scaled Robin condition, but also demonstrate that the results are also applicable to the standard Robin condition, which is not recommended when there is no TCR since large heterogeneity contrast deteriorates its convergence \cite{gander2015optimized,gander2019heterogeneous}.
   \begin{figure}[htbp]
    \centering
    \subfigure{\includegraphics[width=6cm]{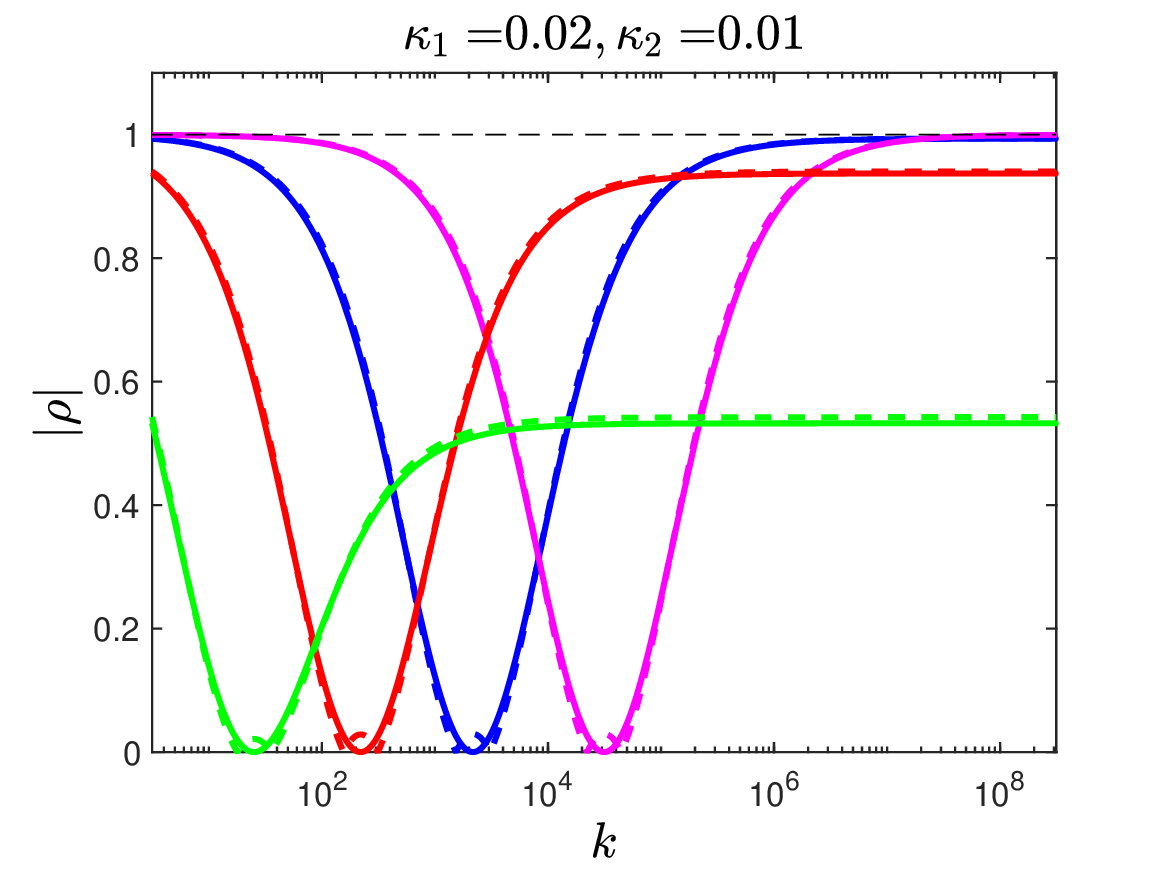}}
    \quad
    \subfigure{\includegraphics[width=6cm]{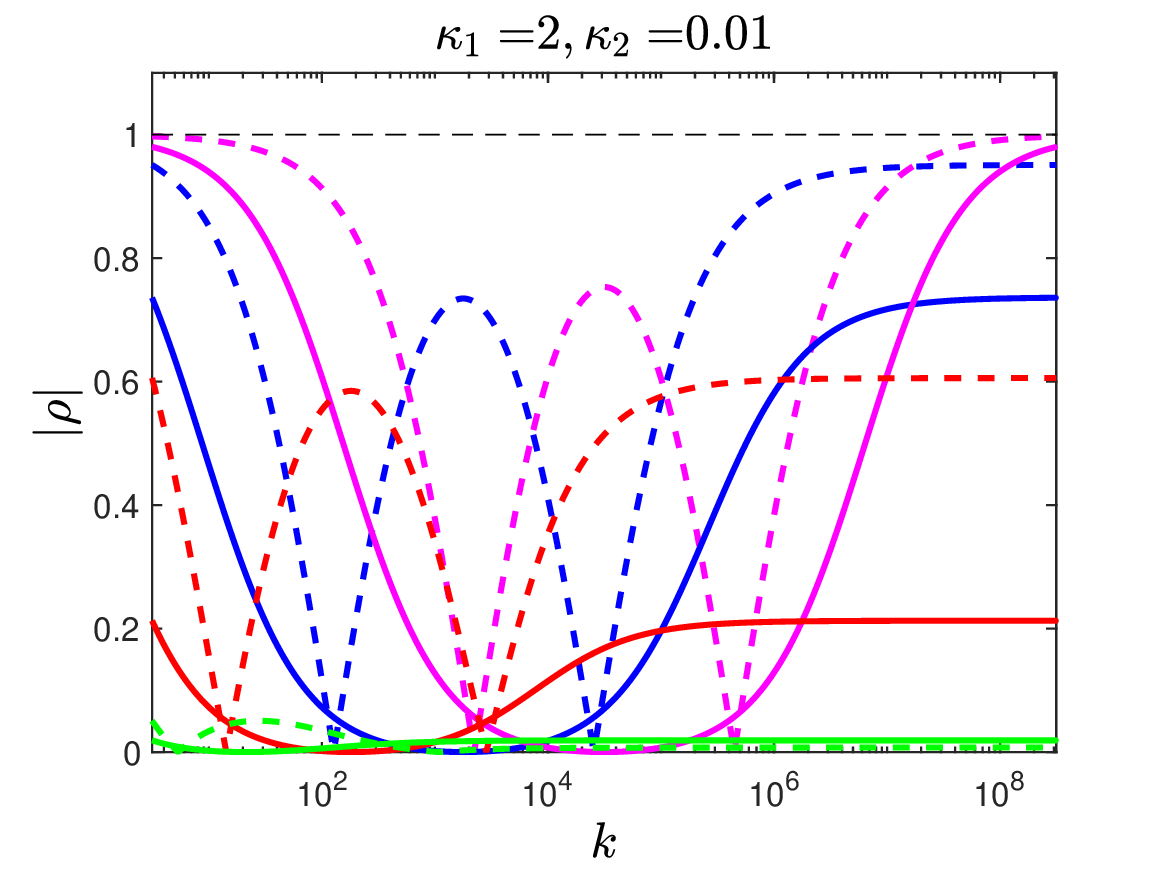}}
    \caption{Optimized convergence factors of the standard Robin algorithms (dotted line) and the scaled Robin algorithms (solid line) for $R_c=0$ (magenta line), $R_c=10^{-4}$ (blue line), $R_c=10^{-2}$ (red line), and $R_c=1$ (green line), meshsize $h=10^{-8}$.}\label{fig4.1}
    \end{figure}
 \end{remark}

\section{Numerical experiments}
To effectively illustrate our theoretical findings, we directly solve the error equation, specifically, the heat conduction problem \eqref{2.1}-\eqref{2.2} with zero source $f=0$. The domain  $\Omega=(0,1)^2$ consists of two subdomains 
$\Omega_1=(-1,0)\times(0,1)$ and $\Omega_2=(0,1)\times(0,1)$, and the interface $\Gamma =\{0\}\times (0,1)$.
The problem is solved using the OSM algorithm \eqref{2.4}-\eqref{2.7}, where the sub-problems are discretized employing the standard centered finite differences with a uniform grid of size $h$. 
To start the subdomain iteration, random initial guesses of $h_{1,2}^0$ on the interface $\Gamma$ are applied to allow the presence of all Fourier frequencies. The iteration terminates when the maximum errors are reduced to a given tolerance $tol$ 
\begin{equation}\nonumber
\max\{\|T_1^n\|_{\infty},\|T_2^n\|_{\infty}\}<tol,
\end{equation}
where $T_i^n$ are the numerical solutions at $n$-th iteration, representing the error since the exact solution is zero. 
Note that for comparison we also show the results of standard Robin condition, whose optimal parameters are obtained numerically by the Nelder-Mead algorithm (the Matlab command {\textbf{fminsearch}}).

\subsection{Effect of the presence of TCR.}\label{subsection5.1}
Our analysis demonstrates that the presence of $R_c$ introduces several distinct properties compared to the scenario without TCR ( $R_c=0$). 
We here illustrate the novel findings resulting from the presence of TCR ($R_c>0$).
\subsubsection{Mesh-independent convergence rates.}
To illustrate the mesh-independent convergence rate estimate in Theorem \ref{theorem3.16}, we run for various meshsizes $h$ the OSM algorithm \eqref{2.4}-\eqref{2.7} until the infinite error reaches  $tol=10^{-6}$ and collect the corresponding results in Table \ref{table_iter}. It shows that the number of iterations required for each run  remains constant regardless of the meshsize. 
That is to say, both Robin transmission conditions exhibit mesh-independent convergence performance. In addition, it is clear that the scaled Robin outperforms the standard Robin. 
\begin{table}[htbp]
\caption{Number of iterations required by the OSM algorithm \eqref{2.4}-\eqref{2.7} to reach the tolerance of $10^{-6}$. Physical parameters: $\kappa_1=2,\kappa_2=0.01,R_c=0.01$.}\label{table_iter}
\begin{center}
\vspace{5pt} 
\renewcommand\arraystretch{1.2}
\begin{tabular}{|c|c|c|c|c|c|}
\hline
\multirow{2}{*}{Trans. Cond.}&\multicolumn{5}{c|}{$h$} \\
\cline{2-6} 
 & $\frac{1}{32}$ & $\frac{1}{64}$&    $\frac{1}{128}$&$\frac{1}{256}$ &$\frac{1}{512}$\\
\hline
standard Robin & $21$ &$21$  &$21$&$21$&$21$  \\
\hline
scaled Robin& $9$ &$9$  &$9$&$9$&$9$\\
 \hline
\end{tabular}
\end{center}
\end{table}
\subsubsection{Effect of heterogeneity contrast.}
To illustrate the effect of heterogeneity contrast as discussed in Remark \ref{remark4.8}, we vary the heterogeneity contrast $\lambda=\frac{\kappa_1}{\kappa_2}$ by adjusting $\kappa_1$ while keeping $\kappa_2$ fixed, and collect in Table \ref{table_convergence} the number of iterations required by our algorithm \eqref{2.4}-\eqref{2.7} to reach a tolerance of $10^{-6}$. The results presented in Table \ref{table_convergence} indicate that, for both Robin conditions, an increased heterogeneity contrast correlates with faster convergence, consistent with the analyses in Remark \ref{remark4.8}. 
Notably, the scaled Robin condition demonstrates performance akin to the case without TCR ($R_c=0$), as described in \cite{gander2015optimized}. 
Conversely, the standard Robin condition exhibits a markedly different behavior compared to the scenario without TCR in \cite{gander2015optimized}, where it requires an increasing number of iterations. 
Particularly, for very high heterogeneity contrasts, the advantages of the scaled Robin condition over the standard Robin diminish significantly.

\begin{table}[htbp]
    \caption{Number of iterations required by the OSM algorithm \eqref{2.4}-\eqref{2.7} to reach a tolerance of $10^{-6}$ for different heterogeneity ratios. Parameter setting: $\kappa_2=
R_c=0.01$, $h=\frac{1}{512}$.}\label{table_convergence}
     \begin{center}
      \vspace{5pt} 
    \renewcommand\arraystretch{1.2}
\begin{tabular}{|c|c|c|c|c|c|}
\hline
\multirow{2}{*}{Trans. Cond.}&\multicolumn{5}{c|}{ $\lambda$} \\
\cline{2-6} 
 & $2$ &$2\cdot10^1$&$2\cdot10^2$&$2\cdot10^3$ & $2\cdot10^4$\\
\hline
standard Robin & $69$ &$28$  &$21$ &$9$&$4$  \\
\hline
scaled Robin& $67$ &$19$  &$6$&$3$&$2$\\
 \hline
\end{tabular}
\end{center}
\end{table}

\subsubsection{Effect of thermal conductivity.}
Remark \ref{remark4.8} indicates that large thermal conductivity also enhances the convergence speed of the algorithm \eqref{2.4}-\eqref{2.7}, even when the heterogeneity contrast is fixed.
To demonstrate this result, we fix $\lambda=200$ while varying the thermal conductivity $\kappa_i,i=1,2$ and collect in Table \ref{table_kappa} the number of iterations required by the algorithm \eqref{2.4}-\eqref{2.7} to reach a tolerance of $10^{-6}$.
The results in Table \ref{table_kappa} clearly show, for both Robin conditions, that the larger the thermal conductivity $\kappa_i,i=1,2$ are, the faster the OSMs converge. 
Especially, for very large thermal conductivity, such as $\kappa_1=2000, \kappa_2=10$, the OSMs for both Robin conditions converge in only two iterations, which is optimal, as shown in \cite{gander2006optimized}.

 \begin{table}[htbp]
\caption{Number of iterations required by the algorithm \eqref{2.4}-\eqref{2.7} to reach a tolerance of $10^{-6}$ for $\lambda=200$, $R_c=0.01$ and $h=\frac{1}{512}$.}\label{table_kappa}
     \begin{center}
      \vspace{5pt} 
    \renewcommand\arraystretch{1.2}
\begin{tabular}{|c|c|c|c|c|}
\hline
\multirow{2}{*}{Trans.  Cond.}&\multicolumn{4}{c|}{$(\kappa_1,\kappa_2)$} \\
\cline{2-5} 
 & $(2,0.01)$  & $(2,0.01)\cdot10^1$&$(2,0.01)\cdot10^2$&$(2,0.01)\cdot10^3$\\
\hline
standard Robin & $21$ &$8$  &$4$&$2$  \\
\hline
scaled Robin& $6$ &$5$  &$4$&$2$\\
 \hline
\end{tabular}
\end{center}
\end{table}

\subsection{Effect of the magnitude of TCR.}
In subsection \ref{subsection5.1} we have numerically investigated the impact of the presence of TCR on the performance of the OSMs. 
Here, we further examine how the magnitude of TCR influences the convergence behavior of the algorithm. 
Remark \ref{scaled_rc} indicates that larger TCR $R_c$ leads to smaller convergence factors. 
To delineate this finding, we record the number of iterations required by the algorithm \eqref{2.4}-\eqref{2.7} for different $R_c$ in Table \ref{table_rc}. 
The results are very clear: the larger the TCR $R_c$ is, the faster the method converges, regardless of the type of Robin transmission condition employed. 
The advantage of the scaled Robin algorithm is more notable for small $R_c$ and diminishes as $R_c$ increases.
Furthermore, a large TCR also facilitates optimal convergence behavior, allowing the algorithm to converge in two iterations.

\begin{table}[htbp]
\caption{Number of iterations required by the algorithm \eqref{2.4}-\eqref{2.7} to reach a tolerance of $10^{-6}$ for $\kappa_1=2,\kappa_2=0.01$ and $h=\frac{1}{512}$.}\label{table_rc}
     \begin{center}
      \vspace{5pt} 
    \renewcommand\arraystretch{1.2}
\begin{tabular}{|c|c|c|c|c|c|c|c|}
\hline
\multirow{2}{*}{Trans.  Cond.}&\multicolumn{7}{c|}{$R_c$} \\
\cline{2-8} 
 &$10^{-3}$ &$10^{-2}$ &$10^{-1}$ &$10^{0}$ &$10^{1}$&$10^{2}$&$10^{3}$\\
\hline
standard Robin & $31$ &$21$  &$10$&$5$ &$3$ &$3$ &$2$\\
\hline
scaled Robin& $6$ &$6$  &$5$&$4$&$3$ &$3$ &$2$\\
 \hline
\end{tabular}
\end{center}
\end{table}
\subsection{Influence of physical parameters on the scaled Robin parameter prediction.}
In this subsection, we investigate how well the theoretical results predict the optimal parameters to be used in the numerical setting, and find out better choices when the asymptotic prediction is affected by physical parameters. 

We vary the Robin parameter $p$ for different $R_c$ and record the number of iterations required by the algorithm \eqref{2.4}-\eqref{2.7}, then plot these results in a log-log way in Fig. \ref{fig4.1;}, compared with the  
solution $p^*$ to the min-max problem in equation \eqref{p_01}, as well as the asymptotically optimized parameters outlined in equation \eqref{p_0}.
Fig. \ref{fig4.1;} shows that, for scaled Robin condition, the optimized parameters in equation \eqref{p_01} predict very well the optimal transmission parameters. In contrast, the parameters derived from asymptotic analysis yield satisfactory results only when $R_c$ is relatively large. As  $R_c$ approaches $0$, the accuracy of these predictions diminishes. This is likely due to the necessity of employing a fine grid to attain the asymptotic regime for very small $R_c$.
Furthermore, we observe that as $R_c$ increases, the algorithm becomes more robust: for a larger parameter interval, the number of iterations remains almost as a constant and does not vary significantly.
\begin{figure}[htbp]
    \centering
\subfigure{\includegraphics[width=6cm]{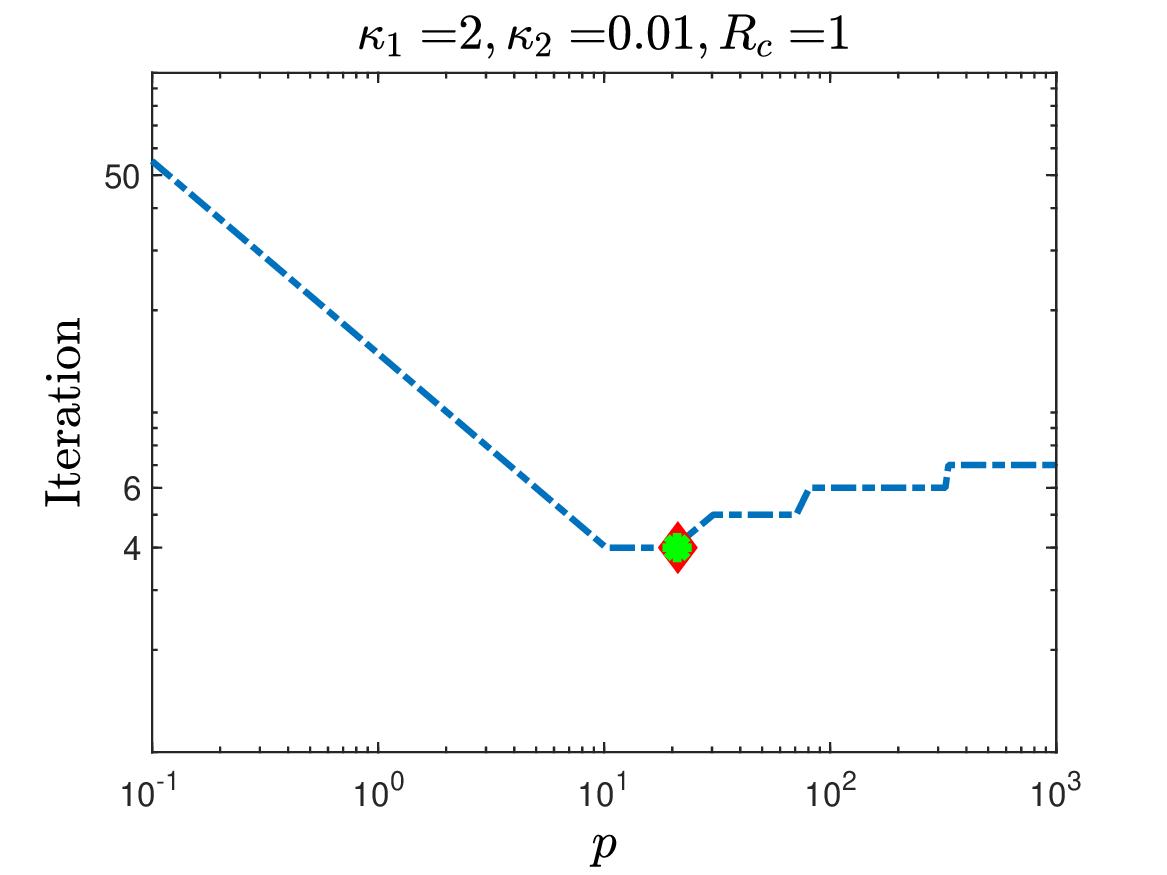}}
  \quad
\subfigure{\includegraphics[width=6cm]{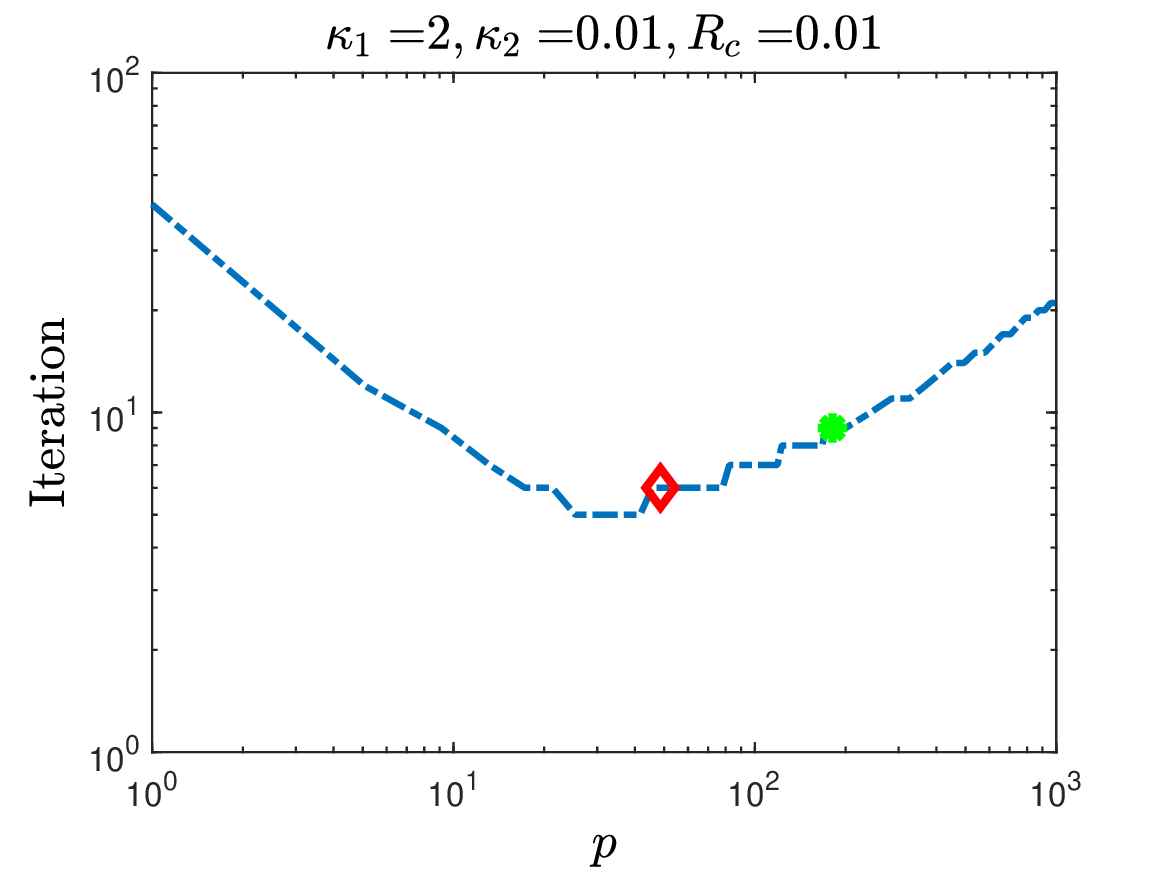}}
  \quad
\subfigure{\includegraphics[width=6cm]{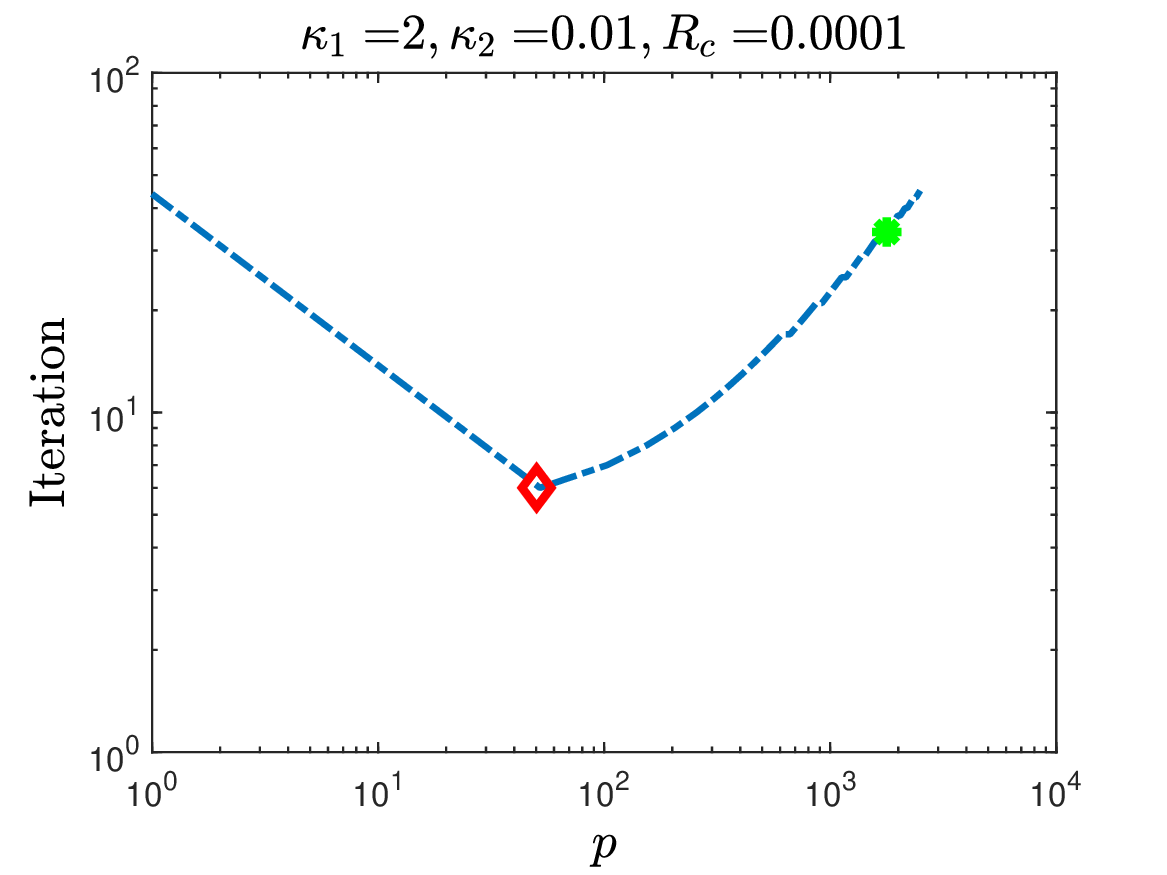}}
\caption{Number of iterations required by algorithm \eqref{2.4}-\eqref{2.7} using meshsize $h=\frac{1}{256}$ as a function of transmission parameter for different $R_c$, compared with the optimized parameters indicated 
by ''\textcolor{green}{$\ast$}'', the asymptotic prediction \eqref{p_0} and ''\textcolor{red}{$\diamond$}'', 
the closed form prediction \eqref{p_01}.}\label{fig4.1;}
    \end{figure}
To study when it reaches the asymptotic regime, we minimize the convergence factor using the Matlab command {\textbf{fminsearch}} for varying meshsize $h$ and record the numerically optimized transmission parameter $p^\ast$. The results are plotted in Fig. \ref{fig5.3}, where the left plot shows the numerically optimized parameter $p^\ast$ as a function of different meshsizes for different $R_c$, while the right plot for different $\kappa_i$. 
From the left plot, one concludes that a small $R_c$ necessitates a correspondingly small meshsize to achieve the asymptotic regime, resulting in a mesh-independent optimized transmission parameter. A similar trend is observed for the thermal conductivities $\kappa_i,i=1,2$, as shown in the right plot.
This behavior can be attributed to the interface condition given in equation \eqref{2.2}, specifically $T_{1}-T_{2}=R_{c}\bm{n_{2}}\cdot(\kappa_{2}\nabla T_{2})$, where the TCR $R_c$ has the same position as the thermal conductivity $\kappa_i$. 
As indicated by \eqref{p_0}, both $R_c$ and $\kappa_i$ manifest simultaneously as $R_c\kappa_i$ in the optimized parameters. We hypothesize that one may employ the mesh-independent prediction from equation \eqref{p_0} when $h<O(R_c\kappa)$ and resort to the mesh-dependent prediction from equation \eqref{p_rc0} when $h>O(R_c\kappa)$, where $\frac{1}{\kappa}=\frac{1}{\kappa_1}+\frac{1}{\kappa_2}$.  
    \begin{figure}[htbp]
    \centering
\subfigure{\includegraphics[width=0.45\textwidth]{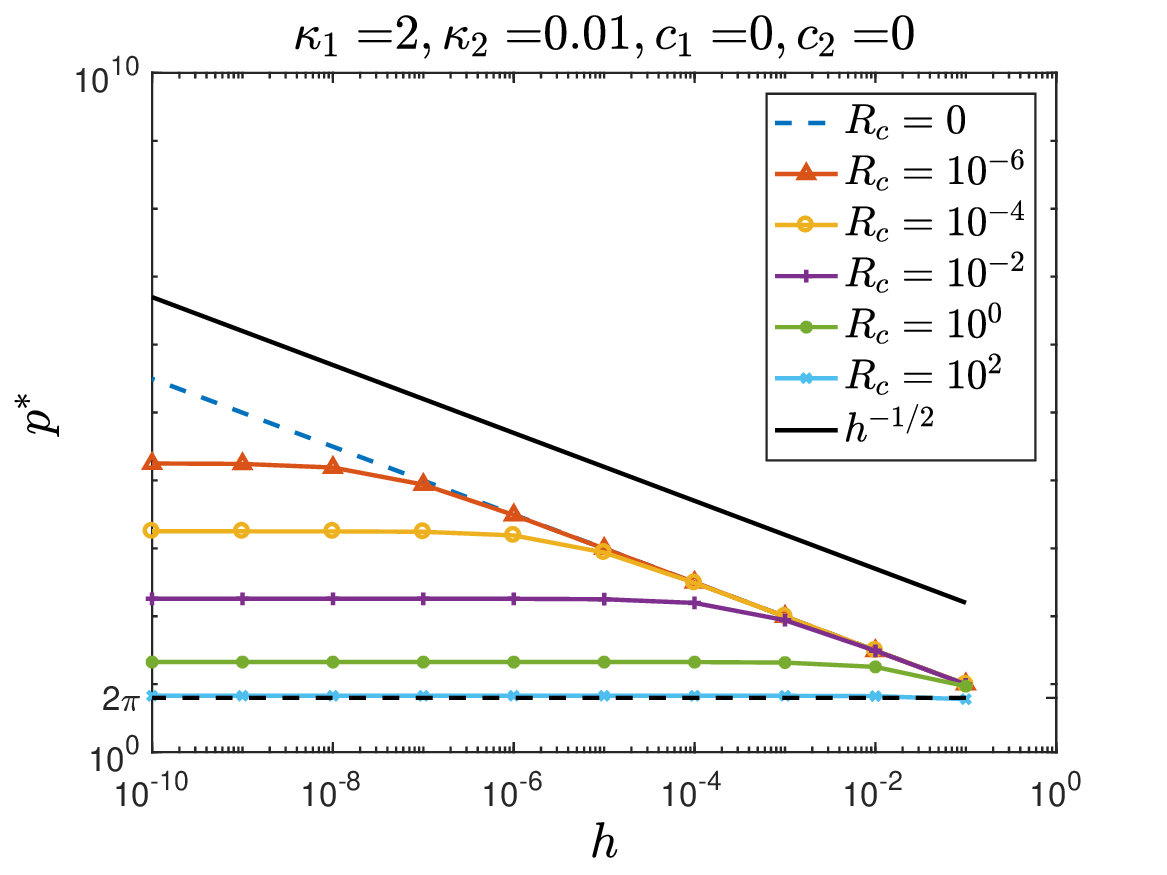}}
\hspace{-0.5cm}\subfigure{\includegraphics[width=0.45\textwidth]{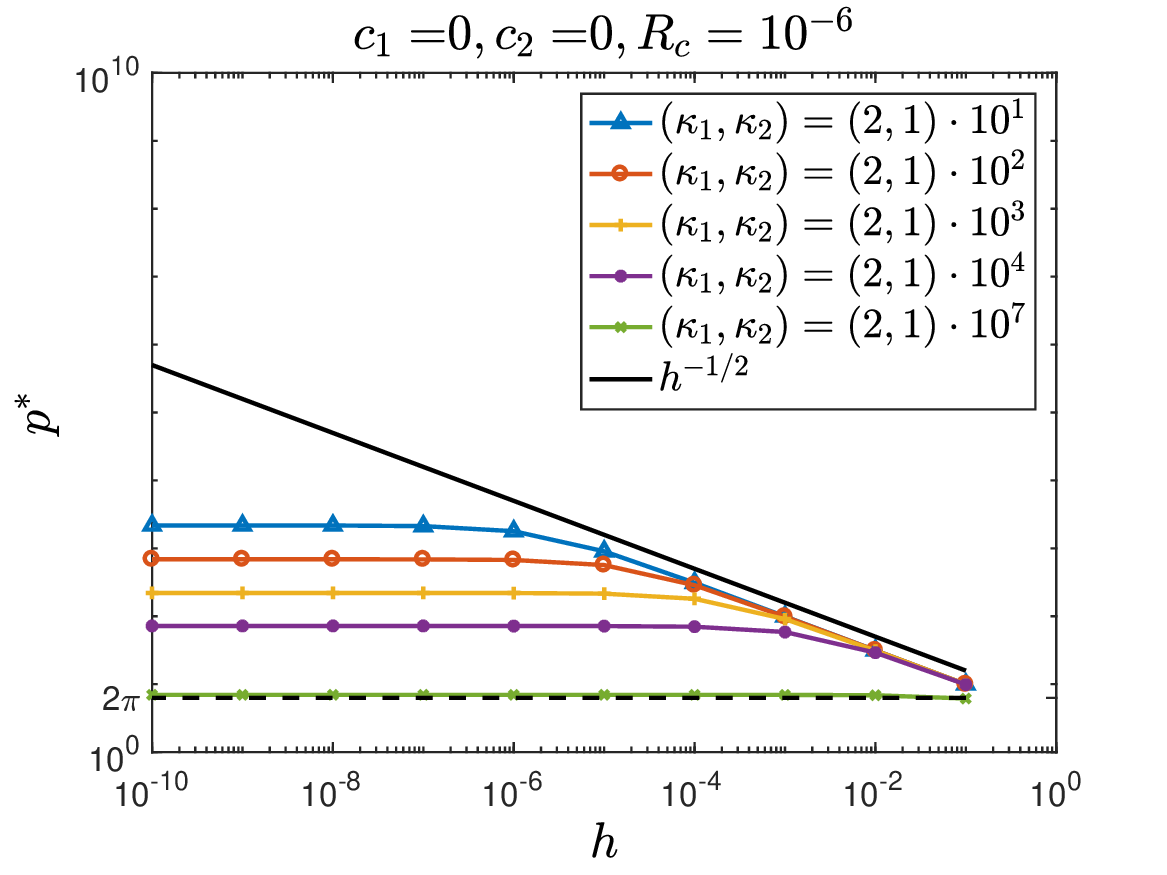}}
       \caption{The numerically  optimized transmission parameter $p^*$ as a function of the meshsize $h$. Left: when the TCR $R_c$ varies; Right: when the thermal conductivities  $\kappa_{1,2}$ vary.
    }\label{fig5.3}
    \end{figure}
\subsection{Application to cylindrical thermal rectifier.}
A classical thermal rectifier consists of two materials with different temperature-dependent thermal conductivities. Between the two materials, there is a small gap that changes due to the thermal expansion and contraction of the materials as the temperature varies. This change in the gap size accounts for the variation in thermal resistance.
When heat flows in the forward direction, the equivalent thermal resistance is low, allowing the heat to flow rapidly. Conversely, in the reverse direction, the equivalent thermal resistance is high, impeding the propagation of heat.

We consider a cylindrical thermal rectifier composed of an internal copper layer and an external alumina layer \cite{zhao2022thermal}. 
The thermal conductivities of both materials are temperature-dependent, given by $\kappa_1(T)=(413.4-0.0516 T) W/mK$ for the copper and $\kappa_2(T)=(5.5+34.5e^{-0.0033(T-237)}) W/mK$ for the  alumina.
The forward and reverse cases of the cylindrical thermal rectifier are shown in Fig. \ref{fig5.5}. 
For the forward case, the hot side is the inner boundary $r_3 = 0.05 m$ of the
structure, while for the reverse case, the hot side is located at the outer boundary $r_1 = 0.15 m$, with the interface set at $r_2 = 0.1 m$. 
The hot and cold sides are maintained at $700 K$ and $300 K$ respectively.
According to the data provided in \cite{zhao2022thermal}, for an initial interface gap of $0.2 mm$, the TCR $R_c=9.280\cdot10^{-5}$ for the forward case  and $R_c=0.01552$ for the reverse case. 
The problem is solved using algorithm \eqref{2.4}-\eqref{2.7} with the scaled Robin transmission condition, where the optimized parameter $p^\ast$ in equation \eqref{p_01} is locally applied according to the local temperature on the interface. The nonlinear sub-problems are discretized using the finite element method with a maximum element size $h_{max}=0.01m$ and solved using a Picard linearization. 
To start the subdomain iteration, random initial guesses of $h_{1,2}^0$ and the temperature $T_i$ on the interface $\Gamma$ are applied. 
The iteration terminates when the difference between two successive steps is reduced to $10^{-6}$ 
\begin{equation}\nonumber
\max\{\|T_1^n-T_1^{n-1}\|_{\infty},\|T_2^n-T_2^{n-1}\|_{\infty}\}<10^{-6},
\end{equation}
where $T_i^n$ are the numerical solutions at $n$-th iteration.
\begin{figure}[htbp]
    \centering
    \begin{minipage}[c]{0.4\textwidth} %minipage使之保持同一行，0.2占这行的0.2
    \centering
    \includegraphics[width=1.1\textwidth]{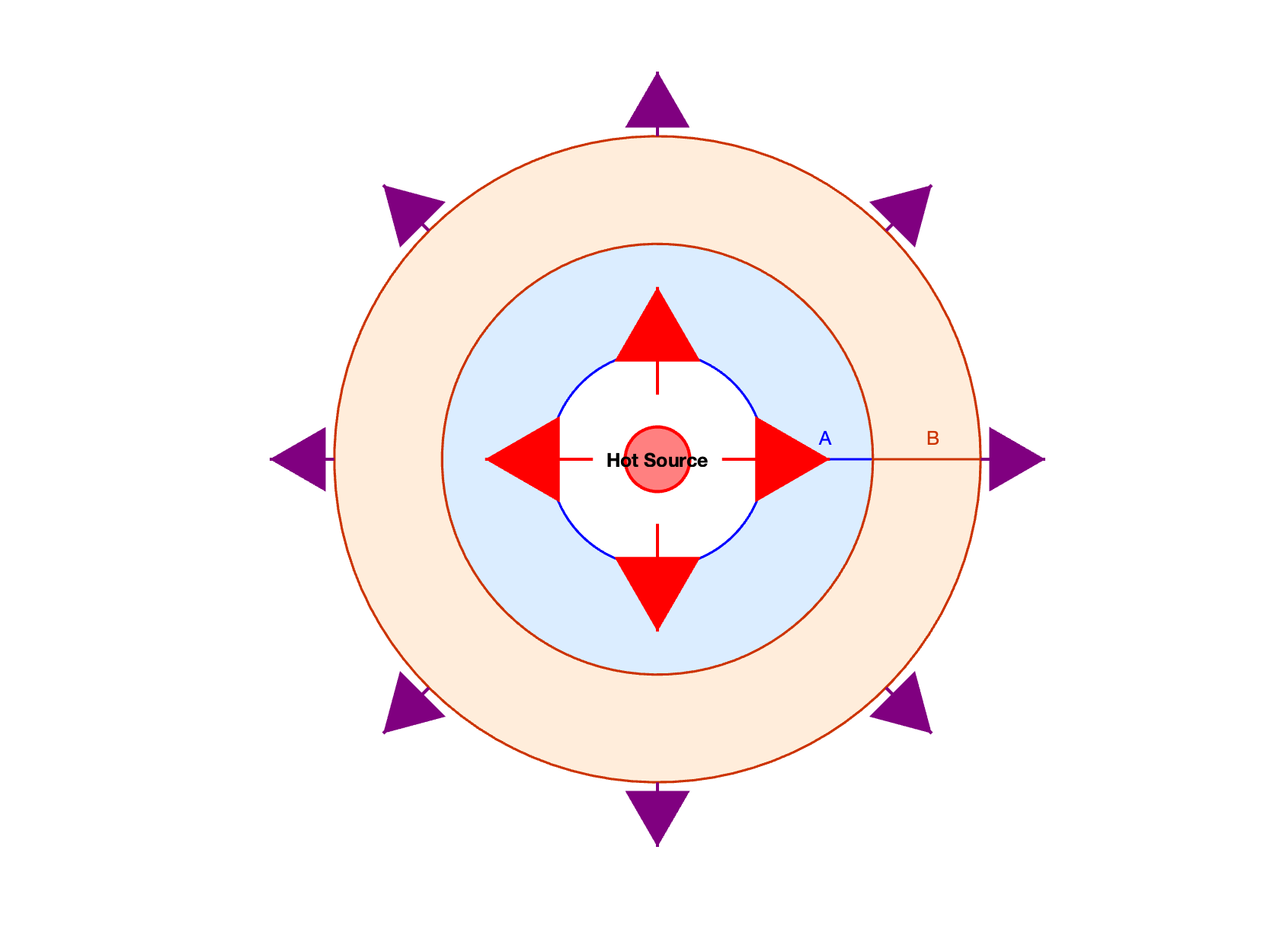} %0.5指图片宽度
    \centerline{(a) Forward case}
  \end{minipage}%
 \quad
  \begin{minipage}[c]{0.4\textwidth}
    \centering
    \includegraphics[width=1.1\textwidth]{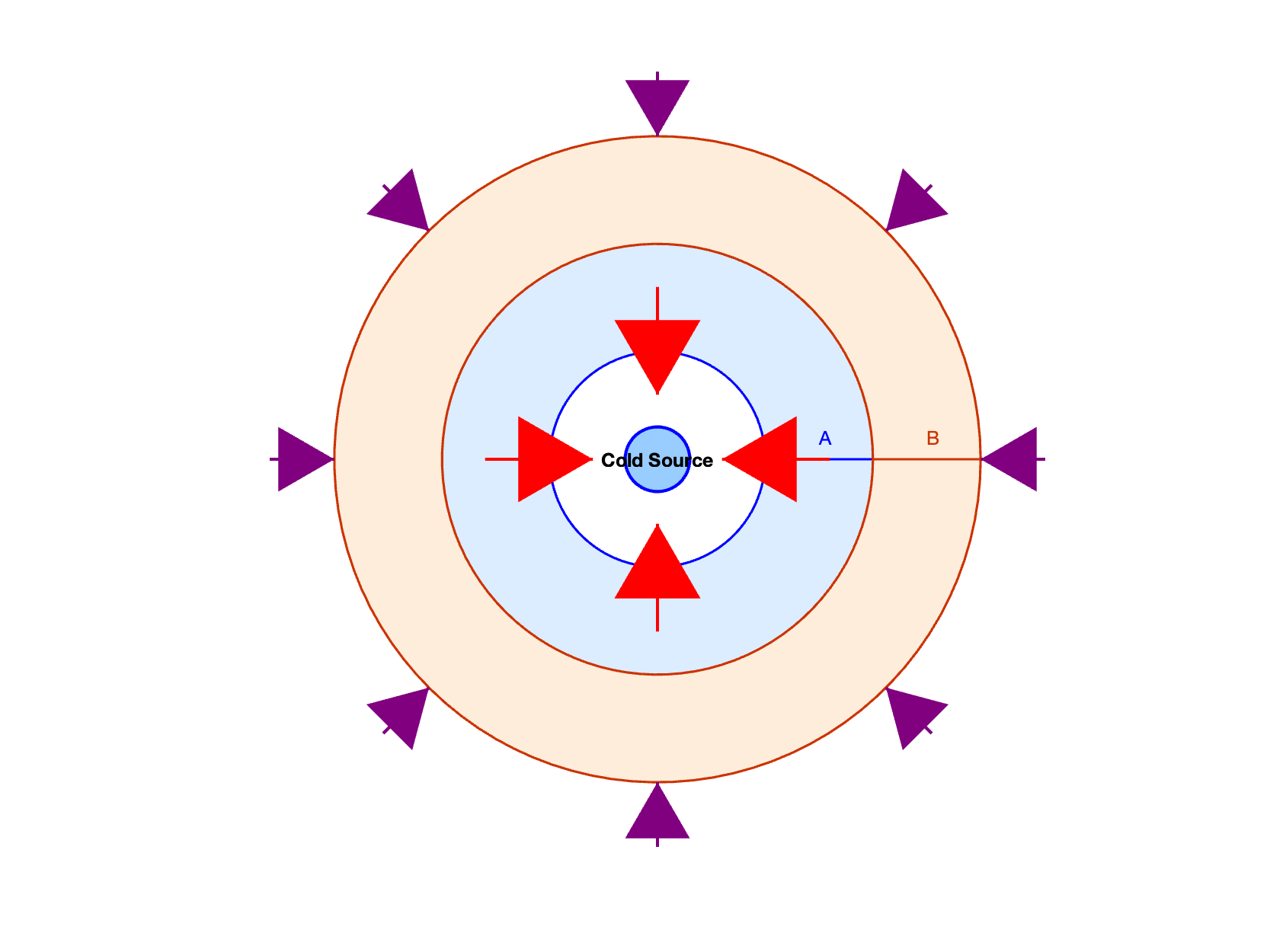}
    \centerline{(b) Reverse case}
     \end{minipage}
       \caption{Cross-section of the composite cylindrical thermal rectifier.
    }\label{fig5.5}
    \end{figure}   
The algorithm requires  $9$ subdomain iterations for the forward case and $5$ iterations for the reverse case. 
The resulting temperature field is illustrated in Fig. \ref{fig5.6}, which clearly demonstrates the function of a thermal rectifier. 
To sum up, this experiment shows that our method is efficient in solving engineering problems, even when the nonlinearity and non-straight interface occur.
    \begin{figure}[htbp]
    \centering
    \begin{minipage}[c]{0.4\textwidth} %minipage使之保持同一行，0.2占这行的0.2
    \centering
    \includegraphics[width=1\textwidth]{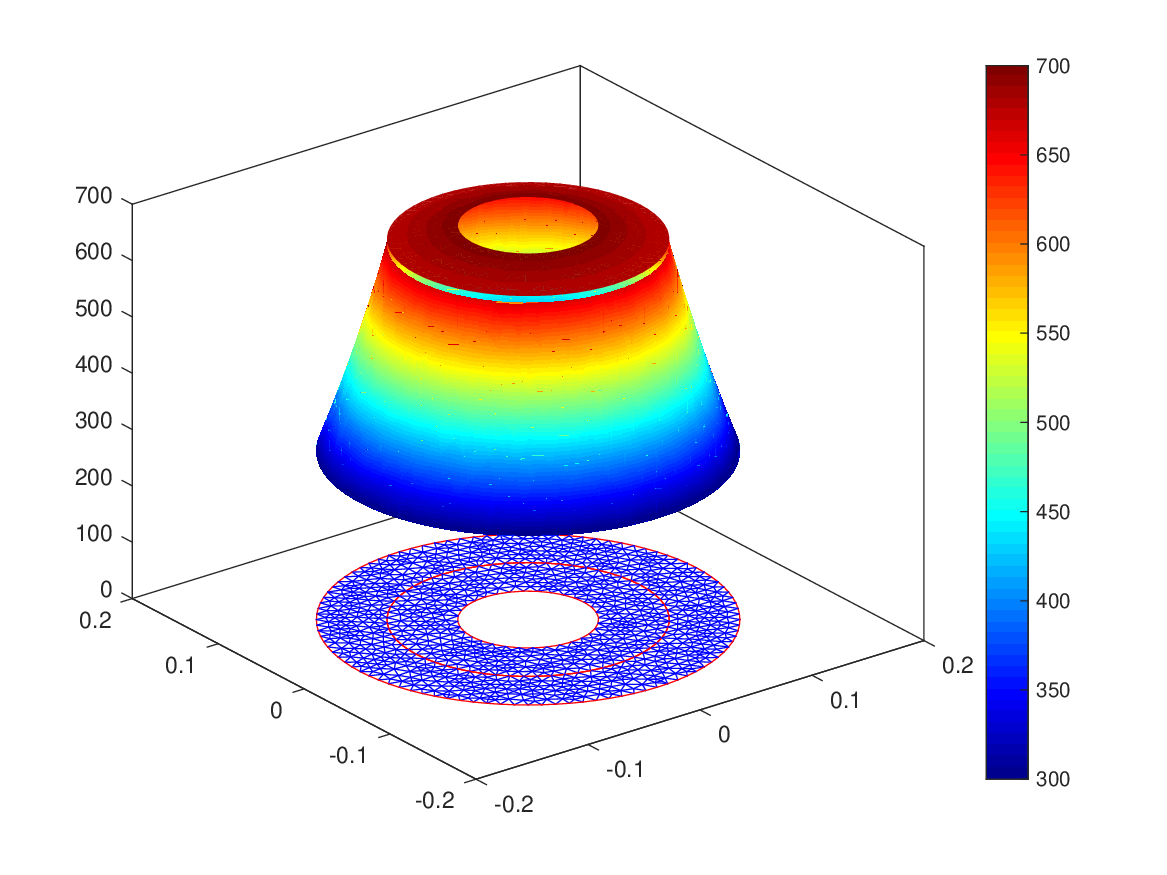} %0.5指图片宽度
    \centerline{(a) Forward case}
  \end{minipage}%
 \quad
  \begin{minipage}[c]{0.4\textwidth}
    \centering
    \includegraphics[width=1\textwidth]{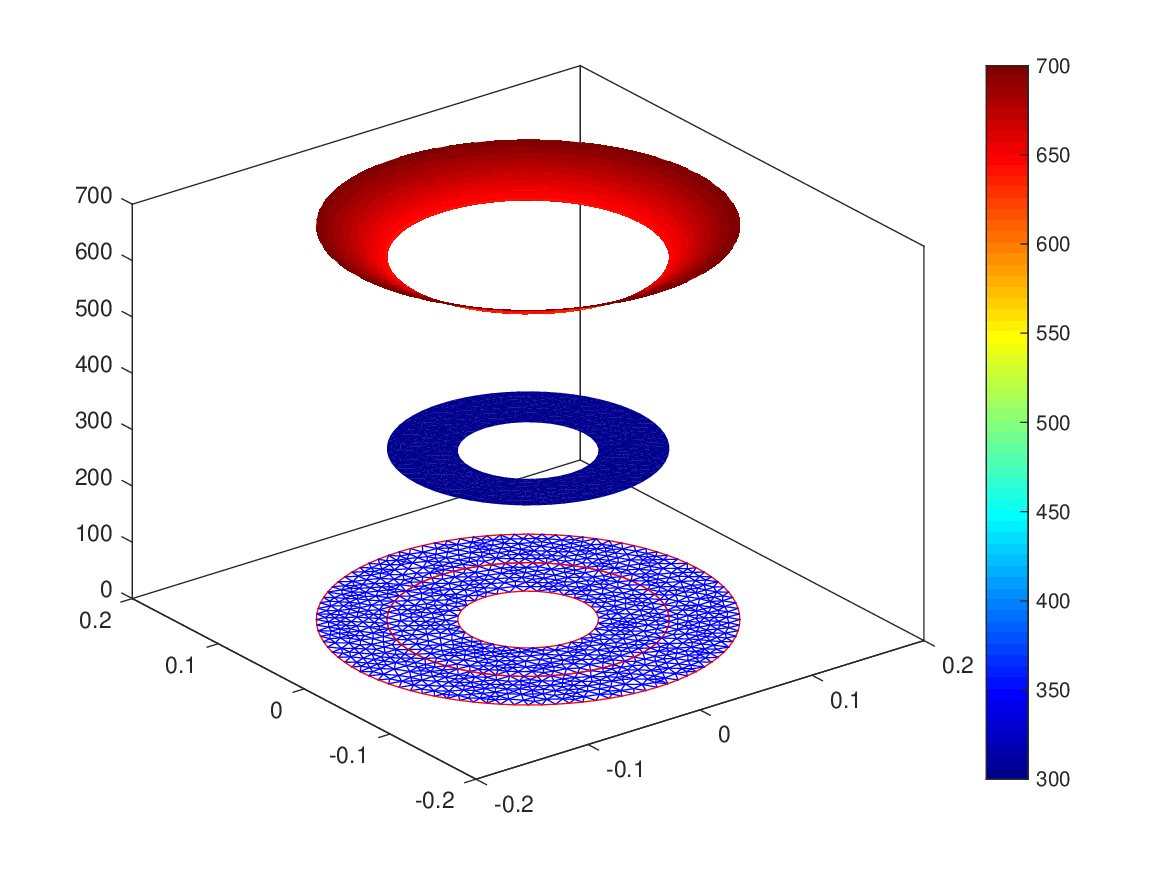}
    \centerline{(b) Reverse case}
     \end{minipage}
       \caption{Temperature field of the thermal rectifier. 
    }\label{fig5.6}
    \end{figure}

 \section{Conclusion}
The OSMs are inherently efficient domain decomposition algorithms for modeling heat conduction in composite materials. 
This study investigates the standard Robin and scaled Robin transmission conditions within the OSM framework, with a particular focus on the role of the TCR, which more accurately represents the physical behavior of composite systems. 
Though the standard Robin condition is quite natural, it is not optimal in terms of both theoretical analysis and numerical performance. In contrast, we concentrate on the OSM with the scaled Robin transmission condition, where the inclusion of TCR leads to significant improvements in algorithm performance. 
Specifically, when TCR is incorporated, the convergence behavior becomes asymptotically mesh-independent, a marked difference from the mesh-dependent convergence observed when TCR is ignored. Furthermore, factors such as large heterogeneity contrast, high thermal conductivity, and a high TCR magnitude all notably accelerate the convergence rate of the scaled Robin OSM algorithm, exhibiting behaviors that substantially differ  from the case without TCR.
Furthermore, this study demonstrates that parameters derived from the equi-oscillation principle yield accurate predictions for the optimal transmission parameters to be applied in the OSM algorithms. However, when the TCR values are exceedingly small, the parameters obtained through asymptotic analysis necessitate the use of excessively fine meshes. This challenge can be addressed by integrating the asymptotic results from the TCR-free case, thereby optimizing the mesh requirements for accurate simulations. The numerical experiments not only support the theoretical findings, but also illustrate the applicability to the nonlinear engineering problems with non-straight interface.

%%%% Acknowledgments %%%%%%%%
\section*{Acknowledgments}
This work is supported in part by
the Jilin Province Science and Technology Development Planning under grant YDZJ202201ZYTS573 and the Jilin Province Department of Education under grant JJKH20250297BS.
%%%% Bibliography  %%%%%%%%%%

% \section*{References}
% \bibliographystyle{elsarticle-num}
% \bibliographystyle{unsrt}
% \bibliographystyle{siam}
 \bibliographystyle{plain}
\bibliography{paper}

\end{document}